\documentclass[12pt,fullpage]{amsart}

\usepackage{amssymb,amsmath,amsthm,amsfonts}
\usepackage[mathscr]{euscript}
\usepackage{dsfont}
\usepackage{graphicx}
\usepackage{cases}
\usepackage{mathrsfs}
\usepackage{esint}
\usepackage{xcolor}

\renewcommand{\leq}{\leqslant}
\renewcommand{\le}{\leqslant}
\renewcommand{\geq}{\geqslant}
\renewcommand{\ge}{\geqslant}

\usepackage{color}
\definecolor{citation}{rgb}{0.2,0.5,0.2}
\definecolor{formula}{rgb}{0.1,0.2,0.5}
\definecolor{url}{rgb}{0,0.2,0.7}
\usepackage[colorlinks=true,linkcolor=formula,urlcolor=url,citecolor=citation]{hyperref}
\allowdisplaybreaks[4]

\newtheorem{theorem}{Theorem}[section]

\newtheorem{lemma}[theorem]{Lemma}
\newtheorem{prop}[theorem]{Proposition}

\newtheorem{claim}[theorem]{Claim}

\theoremstyle{definition}

\theoremstyle{remark}
\newtheorem{rem}[theorem]{Remark}
\theoremstyle{remark}

\numberwithin{equation}{section}

\renewcommand{\tilde}{\widetilde}
\newcommand{\p}[1]{\mathcal{P}_{{#1}}}
\def\R {\mathbb{R}}

\def\N {\mathbb{N}}
\def\S {\mathbb{S}}

\def\O {\Omega}

\def\e {\mathcal{E}}
\def\a {\mathcal{A}}
\def\j {\mathcal{J}}
\def\eps{\varepsilon}
\def\ds{\displaystyle}

\newcommand{\baa}{\begin{array}}
\newcommand{\eaa}{\end{array}}
\newcommand{\ba}{\begin{eqnarray}}
\newcommand{\ea}{\end{eqnarray}}
\newcommand{\be}{\begin{equation}}
\newcommand{\ee}{\end{equation}}

\newlength{\defbaselineskip}
\DeclareGraphicsExtensions{.eps,.pdf,.png}

\begin{document}

\title[]{A counterexample to the Liouville property\\
 of some nonlocal problems}

\author[Julien Brasseur]{Julien Brasseur}
\address{BioSP, INRA, 84914, Avignon, France, and Aix Marseille Univ, CNRS, Centrale Marseille, I2M, Marseille, France}
\author[J\'er\^{o}me Coville]{J\'er\^ome Coville}
\address{BioSP, INRA, 84914, Avignon, France}

\begin{abstract}
In this paper, we construct a counterexample to the Liouville property of some nonlocal reaction-diffusion equations of the form
$$ \int_{\mathbb{R}^N\setminus K} J(x-y)\,( u(y)-u(x) )\mathrm{d}y+f(u(x))=0, \quad x\in\R^N\setminus K,$$
where $K\subset\mathbb{R}^N$ is a bounded compact set, called an "obstacle", and $f$ is a bistable nonlinearity. When $K$ is convex, it is known that solutions ranging in $[0,1]$ and satisfying $u(x)\to1$ as $|x|\to\infty$ must be identically $1$ in the whole space. We construct a nontrivial family of simply connected (non-starshaped) obstacles as well as data $f$ and $J$ for which this property fails.
\end{abstract}

\maketitle

\tableofcontents

\section{Introduction}

\subsection{A nonlocal problem in heterogeneous media}
Let $K$ be a compact set of $\R^N$ with $N\geq2$, and let $\left|\cdot\right|$ be the Euclidean norm in $\R^N$. We are interested in the qualitative properties of positive solutions  $u$ to the  following problem
\begin{align}
\left\{
\begin{array}{rl}
Lu+ f(u)=0  & \text{in }\overline{\R^N\setminus K}, \vspace{3pt}\\
0\leq u \leq 1 & \text{in }\overline{\R^N\setminus K}, \vspace{3pt}\\
u(x)\to1  & \text{as }|x|\to+\infty,
\end{array}
\right.  \label{LIM:u0}
\end{align}
where $f$ is a bistable nonlinearity with $f(0)=f(1)=0$ and $L$ is the nonlocal operator
\begin{equation}\label{DEF:L}
Lu(x):=\int_{\R^N\setminus K} J(x-y)( u(y)-u(x))\mathrm{d}y,
\end{equation}
with $J\in L^1(\R^N)$ a non-negative kernel with unit mass. The precise assumptions on $f$ and $J$ will be given later on.

This type of model naturally arises in the  study of the behavior of  particles  evolving in a heterogeneous medium.
The typical kind of problem we have in mind comes from population dynamics. 
In this setting, the movement of the individuals is modelled by a stochastic process that is defined in a domain that
possesses several inaccessible regions (reflecting the heterogeneity of the environment).
At the macroscopic level, the corresponding density of population $u(t,x)$ satisfies a reaction-diffusion equation that is defined outside a set $K$, which acts as an obstacle. When the individuals follow isotropic Poisson jump processes,  this reaction-diffusion equation is given by
\begin{equation}\label{peq}
\frac{\partial u}{\partial t}=Lu+f(u) \quad \text{in }\R\times \R^N\setminus K,
\end{equation}
and the solutions to \eqref{LIM:u0} are  particular stationary solutions to \eqref{peq}.

In recent years, much attention has been paid to the case of Brownian diffusion. In this situation, the reaction-diffusion equation takes the form
\begin{align}
\label{eqlocale}
\left\{\begin{array}{rl}
\displaystyle\frac{\partial u}{\partial t}=\Delta u+f(u) & \text{in }\R\times\R^N\setminus K,\vspace{3pt}\\
\nabla u\cdot\nu=0 & \text{on }\R\times\partial K.
\end{array}
\right.
\end{align}

This problem was first studied by Berestycki, Hamel and Matano in \cite{BHM}. There, it is shown that there exists a solution
to \eqref{eqlocale} that satisfies $0<u(t,x)<1$ for all $(t,x)\in\R\times\overline{\R^N\setminus K}$,
as well as a classical solution, $u_\infty$, to
\begin{align}
\left\{
\begin{array}{rl}
\Delta u_\infty+f(u_\infty)=0 & \text{in }\overline{\R^N\setminus K},\vspace{3pt}\\
\nabla u_\infty\cdot\nu=0 & \text{on }\partial K, \vspace{3pt}\\
0\le u_\infty\leq 1 & \text{in }\overline{\R^N\setminus K}, \vspace{3pt}\\
u_\infty(x)\to1 & \text{as }|x|\to+\infty.
\end{array} \label{EQ:BHM0}
\right.
\end{align}

This latter solution is actually obtained as the large time limit of $u(t,x)$; more precisely:
$$ u(t,x)\to u_\infty(x)\hbox{ as }t\to\infty,\hbox{ locally uniformly in }x\in\overline{\R^N\setminus K}. $$
In addition, they were able to classify the solutions $u_\infty$ to \eqref{EQ:BHM0} under some geometric assumptions on $K$.
When the obstacle $K$ is either starshaped of directionally convex (see \cite[Definition 1.2]{BHM}), they  prove   that the solutions to \eqref{EQ:BHM0} are actually identically equal to $1$ in the whole set $\overline{\R^N\setminus K}$.
This was further extended to more complex obstacles by Bouhours who showed a sort of ``stability" of this Liouville type property  with respect to small regular perturbations of the obstacle, see \cite{Bouhours}.
From the biological standpoint, this means that, after some large time, \emph{the population tends to occupy the whole space}.

Yet,  when the domain is no longer starshaped nor directionally convex but merely simply connected, it is shown in \cite{BHM} that \emph{this Liouville type property may fail}. In other words, the geometry of the domain may force the population to diffuse heterogeneously in space, even after some large time.

It is expected that \eqref{LIM:u0} and \eqref{EQ:BHM0} share some common properties. In particular, some of the results obtained for \eqref{EQ:BHM0} should, to some extent, hold true as well for \eqref{LIM:u0}.

Recently, Brasseur et al. \cite{BCHV} have shown that \eqref{LIM:u0} enjoys a similar Liouville type property
when $K$ is convex (or close to being convex) and when the data $f$ and $J$ satisfy some rather mild
assumptions.
That is, any solution $u$ to \eqref{LIM:u0} is identically equal to $1$ in the whole set $\overline{\R^N\setminus K}$.
%
They also point out that this cannot be expected for general obstacles since one can easily find counterexamples when $K$ is no longer simply connected. Indeed, take for instance $K=\overline{\mathcal{A}(1,2)}=\overline{B_2}\setminus B_1$ and suppose that $J$ is supported in $B_{1/2}$. Then,  the function $u$ defined by
$$u(x)=\left\{
\begin{array}{cl}
1  & \text{ if }x\in\R^N\setminus B_2, \vspace{3pt}\\
0  & \text{ if }x\in\overline{B_1},
\end{array}
\right.$$
is a continuous solution to \eqref{LIM:u0}; yet, $u$ is not identically $1$ in the whole set $\overline{\R^N\setminus K}$.
In view of this, it is natural to ask:
\begin{align*}
\begin{array}{l}
\textit{what are the optimal geometric assumptions on $K$ ensuring}\\
\textit{that \eqref{LIM:u0} enjoys such a Liouville property?}
\end{array}
\end{align*}

\noindent So far, this question remains open.

In this paper, our main concern is to find out whether it is possible to construct a \emph{nontrivial simply connected} obstacle $K$, as well as data $f$ and $J$, for which \eqref{LIM:u0} has a continuous solution $u$ which is not identically equal to $1$.

Note that this is actually a quite reasonable question. Indeed, since the Liouville property does not hold true on annuli it is quite natural to expect counterexamples on simply connected obstacles which are ``$\eps$-close" to an annulus. We will see that this is indeed the case.  Precisely, we will construct a family of simply connected compact sets $K_\eps$ and data $f_\eps$ and $J_\eps$ for which the solution to \eqref{LIM:u0} need not be identically equal to $1$.

\subsection{Main results}
Before we state our main results, let us first specify the assumptions made all along this paper.
We will assume that $J$ is such that
\be\label{C01}\left\{\baa{l}
J\in L^1(\R^N)\hbox{ is a non-negative, radially symmetric kernel with unit mass},\vspace{3pt}\\
\hbox{there are }0\le r_1<r_2\hbox{ such that }J(x)>0\hbox{ for a.e. }x\hbox{ with }r_1<|x|<r_2,\vspace{3pt}\\
M_1(J):=\displaystyle\int_{\R^N}J(x)|x|\hspace{0.05em}\mathrm{d}x<+\infty\hbox{  and  }J\in W^{1,1}(\R^N),\eaa\right.
\ee
and that $f\in C^1([0,1])$ is a ``bistable" nonlinearity, namely
\be\label{C02}
\left\{
\begin{array}{l}
\exists\,\theta\in(0,1),\ \ f(0)=f(\theta)=f(1)=0,\ \ f<0\hbox{ in }(0,\theta),\ \ f>0\hbox{ in }(\theta,1),\vspace{3pt}\\
\displaystyle\int_{0}^1f(s)\hspace{0.05em}\mathrm{d}s>0,\ \ f'(0)<0,\ \ f'(\theta)>0,\ \ f'(1)<0,\ \ f'<1\text{ in }[0,1].
\end{array}
\right.
\ee
Our first result reads as follows
\begin{theorem}\label{TH:COUNTEREXAMPLE}
Let $N\ge 2$. Then, there are smooth (non-starshaped) simply connected compact obstacles $K$ and data $f$ and $J$ satisfying~\eqref{C01} and~\eqref{C02} for which problem~\eqref{LIM:u0} has a positive nonconstant solution $u\in C(\overline{\R^N\setminus K},[0,1])$.
\end{theorem}

The obstacles constructed in Theorem~\ref{TH:COUNTEREXAMPLE} are  almost of the same nature as those given in~\cite{BHM} for the local case. Namely, we consider an annulus $\mathcal{A}$ into which a small channel is pierced, see Figure~\ref{figKeps} below for a visual illustration.

By contrast with the classical reaction-diffusion, the operator $L$ does \emph{not} enjoy strong compactness properties and has no regularising effects. So our construction is \emph{not} a  simple adaptation of the techniques of proof used for the local problem \eqref{EQ:BHM0}. One of the novelties of this paper is that we show how to circumvent these issues. As we shall explain in the sequel, our argument is in fact general enough to recover the local problem as a limit case (see our remarks below).

Let us briefly describe our approach. Our strategy relies essentially on two ingredients. First, we take advantage of the fact that the kernel $J$ and the nonlinearity $f$ may be chosen at our convenience. That is, instead of considering the problem \eqref{LIM:u0}, we can consider a rescaled version of \eqref{LIM:u0} given an appropriate choice of $J$.
In our setting, $J$ will be such that
\begin{align}\label{J-assum}
J\in L^2(\R^N),~\mathrm{supp}(J) =B_{r}\text{ for some }r>0,\text{ and }J\text{ is radially }\textit{non-increasing.}
\end{align}
Then, given a small parameter $\eps$, we look for a nonconstant positive solution $u_\eps$ to
\begin{align}
\ds{\int_{\R^N\setminus K} J_\eps(x-y)(u_\eps(y)-u_\eps(x))\mathrm{d}y}+ f_\eps(u_\eps(x))=0 \quad\text{in }\overline{\R^N\setminus K}, \label{eq-eps}
\end{align}
that further satisfies $0\leq u_\eps \leq 1$ in $\overline{\R^N\setminus K}$ and $u_\eps(x)\to1$ as $|x|\to+\infty$, where
$$f_\eps(s):=\eps^2f(s)\ \text{ and }\ J_\eps(z)=\frac{1}{\eps^N}J\left(\frac{z}{\eps}\right).$$
In order to prove Theorem \ref{TH:COUNTEREXAMPLE}, we only need to show that, for some $\eps>0$, there is some obstacle $K_\eps$ such that  \eqref{eq-eps} admits a positive nonconstant solution $u_\eps$.

Second, we consider a well-chosen family of smooth simply connected obstacles $(K_\eps)_{0<\eps<1}$ that look like an annulus with a tiny channel of diameter of the order of $\eps^{N/(N-1)}$ pierced in it (see the Figure~\ref{figKeps}). Given such a family, we prove that, for $\eps$ small enough, \eqref{eq-eps} indeed admits a positive nonconstant continuous solution. More precisely, we prove the following
 \begin{theorem}\label{TH:eps}
Let $N\ge 2$. Let $J$ and $f$ be such that \eqref{C01}, \eqref{C02} and \eqref{J-assum} hold. Then, there exist $\eps_*>0$ and a family of smooth simply connected obstacles $(K_\eps)_{0<\eps<1}\subset\R^N$ such that, for all $0<\eps<\eps^*$, there is a positive nonconstant solution $u_\eps\in C(\overline{\R^N\setminus K_\eps},[0,1])$ to \eqref{eq-eps}.
\end{theorem}

Due to the lack of a strong regularising property of \eqref{eq-eps}, the construction of $u_\eps$ relies essentially on elementary arguments. In particular, we obtain  a solution $u_\eps$ to \eqref{eq-eps} using an adequate monotone iterative scheme and elementary estimates. The main difficulty in our proof lies in the construction of an adequate pair of ordered continuous sub- and super-solution in a context where the equation \eqref{LIM:u0} does not allow the use of traditional schemes based on compactness arguments.  To cope with this major difficulty,  we make a detailed construction of the obstacle $K_\eps$ and design it in such a way that we still can obtain standard $L^2$-estimates by elementary means. This requires a detailed analysis  of all the parameters involved at each steps of our construction,  especially  when  we construct our super-solution. To construct our super-solution we rely on the fact that a solution $u_\eps$ to \eqref{eq-eps} satisfies in particular
\begin{equation}\label{eq2-eps}
\frac{1}{\eps^2}\int_{\R^N\setminus K_\eps}J_\eps(x-y)(u_\eps(y)-u_\eps(x))\mathrm{d}x+f(u_\eps(x))=0,
\end{equation}
and, from there, relying essentially on the Bourgain-Brezis-Mironescu characterisation of Sobolev spaces (see e.g. \cite{BBM,Ponce2004}), we can interpret the first term on the left-hand side as a \emph{nonlocal approximation} of $\Delta u$ in the sense that its energy approximates the $L^2$-variation of $u$. This, in turn, with a pertinent choice of $K_\eps$ and a well-chosen auxiliary problem, allows one to derive  \textit{a priori} bounds to construct a super-solution by means of variational methods.


A striking consequence of our construction is that it adapts almost straightforwardly to
other situations. For example, it applies to the standard reaction-diffusion equation
\eqref{EQ:BHM0} providing so an alternative proof of the existence of a counterexample.
But it also extends to broader classes of nonlocal operators where the dispersal process need not be isotropic but instead depends on the geodesic distance between points in $\overline{\R^N\setminus K}$.
Indeed, our proof also adapts (with almost no changes) to operators of the form
\begin{align}
L_{\mathrm{g}} u(x):=\int_{\R^N\setminus K}\widetilde{J}(d_\mathrm{g}(x,y))(u(y)-u(x))\mathrm{d}y, \label{Op:Disp}
\end{align}
where $d_{\mathrm{g}}(\cdot,\cdot)$ is the geodesic distance on $\overline{\R^N\setminus K}$ and $\widetilde{J}\in L_{\mathrm{loc}}^1(0,\infty)$ is such that
\begin{align}
\sup_{x\in\R^N\setminus K}\int_{\R^N\setminus K}\widetilde{J}(d_{\mathrm{g}}(x,y))\hspace{0.05em}\mathrm{d}y<\infty, \label{J:geo}
\end{align}
and $z\mapsto \widetilde{J}(|z|)$ satisfies \eqref{C01}.

More precisely, we have
\begin{theorem}\label{TH:COUNTEREXAMPLE2}
Let $N\ge 2$. Then, there are smooth (non-starshaped) simply connected compact obstacles $K$ and data $f$ and $\widetilde{J}$ satisfying~\eqref{C02} and~\eqref{J:geo} for which the problem
\begin{align}
\left\{
\begin{array}{rl}
L_{\mathrm{g}}u+ f(u)=0  & \text{in }\overline{\R^N\setminus K}, \vspace{3pt}\\
u(x)\to1  & \text{as }|x|\to+\infty,
\end{array}
\right. \label{EQ:geo}
\end{align}
has a solution $u\in C(\overline{\R^N\setminus K},[0,1])$ which is not identically equal to $1$ in $\overline{\R^N\setminus K}$.
\end{theorem}

The obstacle $K$ and the data $f$ and $J\,(=\widetilde{J}(\left|\cdot\right|))$ constructed at Theorem \ref{TH:COUNTEREXAMPLE2} are exactly the same as in Theorem \ref{TH:COUNTEREXAMPLE}.

Problem \eqref{EQ:geo} is of interest in its own right. It gives an alternative way to describe the evolution of particles within a perforated domain which, in some situations, may be regarded as more realistic. The point here is that particles \emph{cannot} travel through $K$ (as is it the case for problem \eqref{LIM:u0}). Instead, they are compelled to ``bypass" $K$ as if it was a material obstacle. This particularity may be helpful to study the dynamics of some species (such as worms or spores) for which this behavior is well-suited.

When needed we will state in side remarks the necessary changes to make to the proofs in order to handle this type of dispersal processes.

\begin{rem}
It turns out that the techniques of proof used in \cite{BCHV} to establish the Liouville property of \eqref{LIM:u0} for convex domains also apply to this modified setting (at least when $J$ is non-increasing), but we leave this to a subsequent paper.
\end{rem}

The paper is organized as follows. After describing our notations, we recall some results from the literature in Section \ref{section-pre}. In Section \ref{section-data}, given a pair $(J,f)$ we construct an adequate family of obstacles.  Then, in Section \ref{section-supersol}, we construct some particular super-solutions to the problem \eqref{eq-eps}.  Finally, in Section \ref{section-proof}, we use the super-solution constructed at Section~\ref{section-supersol} to prove Theorem \ref{TH:eps}.
\subsection*{Notations}

Let us list a few notations that will be used throughout the paper.

As usual, $\S^{N-1}$ denotes the unit sphere of $\R^N$ and $B_R(x)$ the open Euclidian ball of radius $R>0$ centred at $x\in\R^N$ (when $x=0$, we simply write $B_R$).
We denote by $\mathcal{A}(R_1,R_2)$ the open annulus $B_{R_2}\setminus \overline{B_{R_1}}$.

For a compact set $\Omega\subset\R^N$, we denote by $\mathrm{diam}(\Omega)$ its diameter, given by
$$ \mathrm{diam}(\Omega):=\sup_{x,y\in\Omega}|x-y|. $$

The $N$-dimensional Hausdorff measure will be denoted by $\mathscr{H}^{N}$.
For a measurable set $E\subset\R^N$, we denote by $|E|$ its Lebesgue measure and by $\mathds{1}_E$ its characteristic function. If $0<|E|<\infty$ and if $g:\R^N\to\R$ is locally integrable, we denote by
$$ \fint_E g(x)\hspace{0.05em}\mathrm{d}x=\frac{1}{|E|}\int_E g(x)\hspace{0.05em}\mathrm{d}x, $$
the average of $g$ in the set $E$. Also, we denote by $L^p(E)$, $1\leq p \leq\infty$, the Lebesgue space of (equivalence classes of) measurable functions $g$ for which the $p$-th power of the absolute value is Lebesgue integrable when $p<\infty$ (resp. essentially bounded when $p=\infty$).

\section{Preliminaries}\label{section-pre}

In this section, we recall some known results that will be used throughout the paper. In most cases, we will omit their proofs and point the interested reader to the relevant references.

We first state a general existence result.

\begin{lemma}\label{Lem-iter}
Assume that $f$ and $J$ satisfy \eqref{C01} and \eqref{C02}. Let $K\subset \R^N$ be a compact set and let $\underline{u},\overline{u}\in C(\R^N\setminus K)$ be such that
\begin{equation*}
\left\{\baa{ll}
L\overline{u}+f(\overline{u})\leq 0 & \mbox{in } \R^N\setminus K,\vspace{3pt}\\
L\underline{u}+f(\underline{u})\ge 0 & \mbox{in } \R^N\setminus K.\eaa\right.
\end{equation*}
Assume, in addition, that
\begin{align}
\limsup_{|x|\to\infty}\,\underline{u}(x)=\lim_{|x|\to\infty}\overline{u}(x)=1, \label{SousSur:iter1}
\end{align}
and that
\begin{align}
0\leq\underline{u}\le \overline{u}\leq1 \quad{\mbox{in }}\R^N\setminus K. \label{SousSur:iter2}
\end{align}
Then, there exists $u \in L^{\infty}(\R^N\setminus K)$ such that
\begin{equation*}
\left\{\begin{array}{rl}
Lu+f(u)= 0  & {\mbox{in }} \R^N\setminus K,\vspace{3pt}\\
\underline{u}\le u\le \overline{u} & {\mbox{in }} \R^N\setminus K.\end{array}
\right.
\end{equation*}
\end{lemma}
Although the proof of Lemma~\ref{Lem-iter} relies on rather standard arguments it is not that straightforward. For this reason, we will give a detailed proof (which is postponed to the Appendix at the end of the paper).

Next, we recall a regularity result for nonlocal equations of the form
\begin{equation}\label{eq-pre-reg}
\int_{\O\setminus K}J(x-y)\hspace{0.05em}u(y)\hspace{0.05em}\mathrm{d}y - \mathcal{J}(x)\hspace{0.05em}u(x) +f(u(x))=0 \quad\text{in }\O\setminus K,
\end{equation}
where
\begin{equation}\label{defJ(x)}
\mathcal{J}(x):=\int_{\R^N\setminus K}J(x-y)\hspace{0.05em}\mathrm{d}y.
\end{equation}
Precisely,
\begin{lemma}\label{LEMMA:INT}
Assume that $f\in C^{1}([0,1])$ and that $J$ satisfies \eqref{C01}.  Let $\O\subset \R^N$ be an open set having $C^1$ boundary. Suppose that $K\subset\O$ is a compact set and that
\begin{align}
\max_{[0,1]}f'<\inf_{\O\setminus K}\,\mathcal{J}.   \label{monotonie2}
\end{align}
Let $u\in L^\infty(\O\setminus K,[0,1])$ be a solution to \eqref{eq-pre-reg} a.e. in $\O\setminus K$. Then, $u$ can be redefined up to a negligible set and extended as a uniformly continuous function in $\overline{\O\setminus K}$.
\end{lemma}
For a detailed proof, we refer to \cite[Lemma 3.2]{BCHV} (see also \cite{Bates,BR}).
\begin{rem}
Note that $\O$ need not be bounded. In particular, Lemma~\ref{LEMMA:INT} holds when $\O=\R^N$.
\end{rem}
Finally, we recall the following result
\begin{lemma}\label{LEMMA:Behave}
Let $K\subset\R^N$ is a compact set and suppose that $f$ and $J$ satisfy \eqref{C01} and \eqref{C02}. Assume further that $J$ is compactly supported and that $J\in L^2(\R^N)$. Let $u\in C(\R^N\setminus K,[0,1])$ be a solution to
\be\left\{\begin{array}{rcl}
Lu+f(u) & \!\!=\!\! & 0\ \hbox{ in }\,\overline{\R^N\setminus K},\vspace{3pt}\\
\ds{\sup_{\R^N\setminus K}}\,u &  \!\!=\!\! & 1,\end{array}\right.  
\ee
Then, $u(x)\to1$ as $|x|\to \infty$.
\end{lemma}
The proof may be found in \cite[Lemma 7.2]{BCHV}.

\begin{rem} \label{rem-geod}
The above results still hold when $J(x-y)$ is replaced by $\widetilde{J}(d_{\mathrm{g}}(x,y))$. For the validity of Lemma~\ref{Lem-iter} in this case, we refer to Remark~\ref{geod-montone-iterative} in the Appendix.
On the other hand, a careful inspection of the proof of \cite[Lemma 3.2]{BCHV} shows that the condition (2.6)
with $\mathcal{J}$ replaced by
\begin{align}
\widetilde{\mathcal{J}}(x):=\int_{\R^N\setminus K}\widetilde{J}(d_{\mathrm{g}}(x,y))\mathrm{d}y, \label{defJtilde(x)}
\end{align}
still implies the continuity of solutions to
$$ \int_{\Omega\setminus K}\widetilde{J}(d_\mathrm{g}(x,y))\hspace{0.05em}u(y)\hspace{0.05em}\mathrm{d}y-\widetilde{\mathcal{J}}(x)\hspace{0.05em}u(x)+f(u(x))=0, $$
in $\overline{\Omega\setminus K}$. Similarly, Lemma~\ref{LEMMA:Behave} holds as well with $L_\mathrm{g}$ (as given by \eqref{Op:Disp})
instead of $L$ since its proof requires only estimates on convex regions on which it
trivially holds that $d_\mathrm{g}(x,y)=|x-y|$.
\end{rem}

\section{Construction of a family of obstacles}\label{section-data}

This section is devoted to the construction of an appropriate family of obstacles $(K_\eps)_{0<\eps<1}$.
Our construction will depend on the interplay with the datum $(J,f)$.
As mentioned in the introduction, we will assume that $J$ satisfies \eqref{C01} and \eqref{J-assum} and that $f$ satisfies \eqref{C02}.
However, before constructing $(K_{\eps})_{0<\eps<1}$, we need to define some important quantities depending on $f$ and $J$.
We will call $C_0>0$ and $M_2(J)>0$ the constants respectively defined by
\begin{numcases}{~}
C_0:= \displaystyle\max_{s\in[0,1]} f(s), & \label{C_0}\\
M_2(J):=\displaystyle\int_{\R^N}J(z)|z|^2\mathrm{d}z. & \label{M2}
\end{numcases}
\noindent Note that the assumptions \eqref{C01} and \eqref{C02} guarantee that these two numbers are well-defined. Furthermore, we introduce two quantities, $C_{N,J}$  and $R_0^* $, respectively defined by
\begin{numcases}{~}
C_{N,J}:=\frac{\pi^2\hspace{0.05em}M_2(J)}{32\hspace{0.05em}N}, & \vspace{3pt} \label{CNJ}\\
R^*_0(J,f):=\sqrt{\frac{\theta C_{N,J}}{5C_0}}. & \label{R0}
\end{numcases}

\begin{figure}
\centering
\includegraphics[scale=0.35]{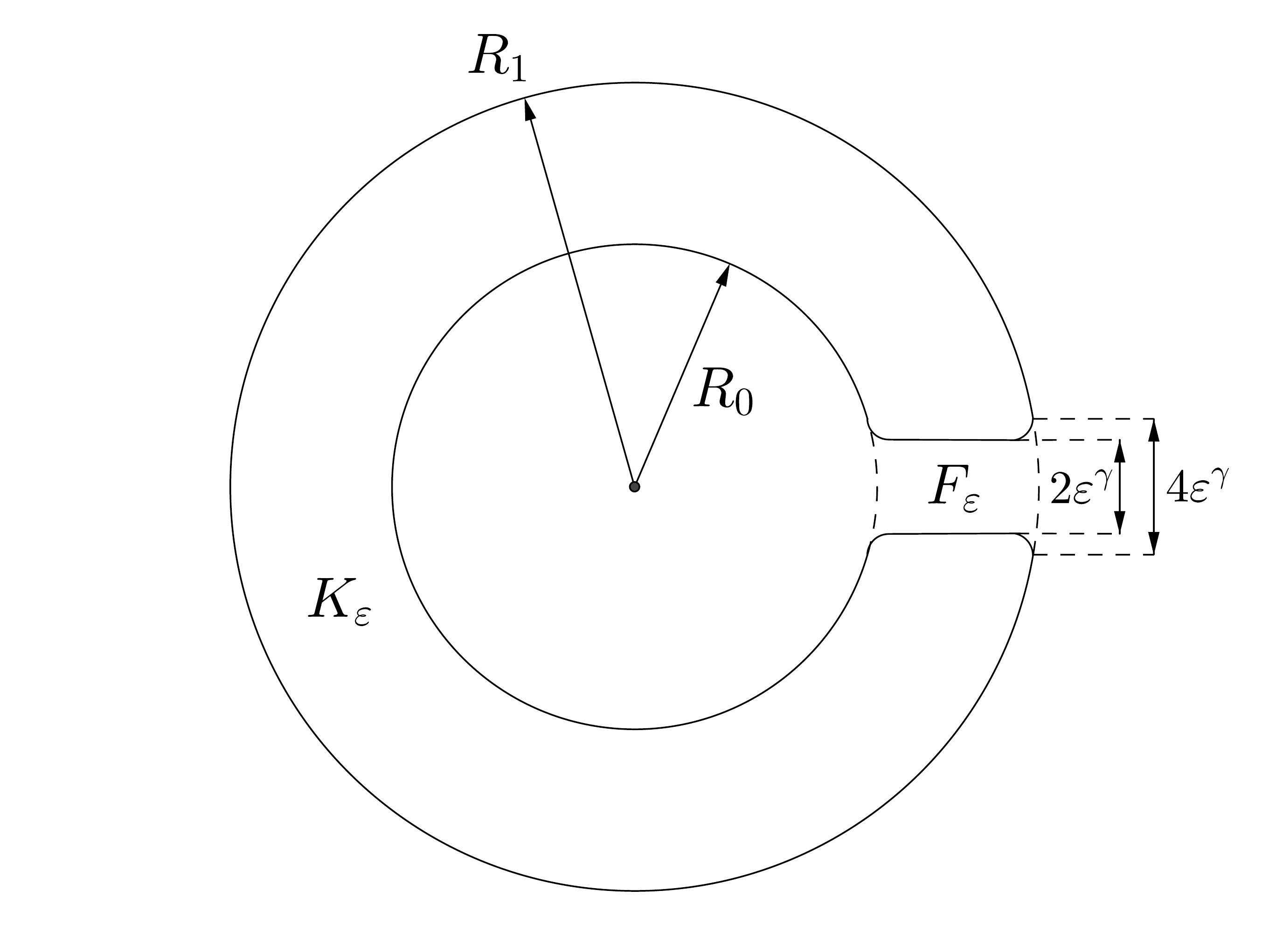}
\caption{Illustration of $K_\eps$ in dimension $2$.}
\label{figKeps}
\end{figure}

Let us now start the construction of the obstacle. Fix some $R_1>2$ and let $0<R_0<R_0^*(J,f)$ (where $R_0^*(J,f)$ is as in \eqref{R0}). Let $0<\eps<1$ be a small parameter and set $\gamma:=\frac{N}{N-1}$.
We call $\mathcal{A}$ the annulus $\mathcal{A}:=\mathcal{A}(R_0,R_1)$ and we consider a smooth compact simply connected set $K_\eps\subset\overline{\mathcal{A}}$ satisfying the following properties:
\begin{enumerate}
\item[(i)] $\overline{\mathcal{A}}\cap\big\{x\in\R^N\,;\ x_1\leq0\big\}\subset K_\eps$, \vspace{3pt}
\item[(ii)] $\overline{\mathcal{A}}\cap\big\{x\in\R^N\,;\ x_1>0,|x'|>2\eps^\gamma\big\}\subset K_\eps$, \vspace{3pt}
\item[(iii)] $K_\eps\subset\left(\overline{\mathcal{A}}\cap\big\{x\in\R^N\,;\ x_1\leq0\big\}\right)\cup\left(\overline{\mathcal{A}}\cap\big\{x\in\R^N\,;\ x_1>0,|x'|\geq\eps^\gamma\big\}\right)$, \vspace{3pt}
\item[(iv)] $\mathcal{A}(R_0+\eps^\gamma/4,R_1-\eps^\gamma/4)\cap\big\{x\in\R^N;x_1>0,|x'|\geq\eps^\gamma\big\}\subset K_\eps$.
\end{enumerate}
where $x=(x_1,x')$ and $x'=(x_2,\ldots,x_N)$ (see Figure~\ref{figKeps} for a visual illustration).
Furthermore, we define the following open set:
$$ F_\eps:=\mathcal{A}\setminus K_\eps. $$
We will refer to $(K_\eps)_{0<\eps<1}$ as the \emph{family of obstacles associated to the pair} $(J,f)$.

Let us also list in this section a preparatory lemma.
\begin{prop}\label{LE:CONTINUE}
Let $N \ge 2$. Suppose that $f$ and $J$ are such that \eqref{C02} and \eqref{J-assum} hold true. Assume further that $(K_\eps)_{0<\eps<1}$ is the family of obstacles associated to the pair $(J,f)$. Let
\begin{align}
f_\eps(s):=\eps^2f(s)\ \text{ and }\ J_\eps(z):=\frac{1}{\eps^N}\hspace{0.05em}J\left(\frac{z}{\eps}\right). \label{Donnees:fJ}
\end{align}
Then, there exists some $\eps_0>0$ depending only on $N$, $R_0$, $J$ and $f'$, such that
\begin{align}
\max_{[0,1]}f_\eps'<\inf_{x\in\R^N\setminus K_\eps}\int_{\R^N\setminus K_\eps}J_\eps(x-y)\hspace{0.05em}\mathrm{d}x, \hbox{ for all }\eps\in(0,\eps_0). \label{CONTINUE}
\end{align}
\end{prop}
Proposition \ref{LE:CONTINUE} will play an important role in the sequel. Inter alia, it guarantees that the solutions of some nonlocal equations defined in the sequel are \emph{continuous}.

\begin{proof}
By assumption \eqref{J-assum}, up to rescale $J$, we may assume without loss of generality that
$$ \mathrm{supp}(J)=B_{1/2}. $$
Let $0<\eps<\eps_1:=\min\{1,R_0/2\}$ and $x\in \R^N\setminus K_\eps$. Define
$$ \widetilde{F}_\eps:=\{z\in \R^N;R_0<z_1<R_1,\,|z'|<\eps^\gamma\} \hbox{ and } \Lambda_\eps(x):=B_{\eps/2}\cap(\widetilde{F}_\eps-x). $$
We will estimate from below the integral in the right-hand side of \eqref{CONTINUE}. For it, we will treat separately the case where $x\in\widetilde{F}_\eps$ and the case where $x\in\R^N\setminus(K_\eps\cup\widetilde{F}_\eps)$.
\vskip 0.3cm

\noindent\emph{Step 1: Lower bound in $\widetilde{F}_\eps$}
\vskip 0.3cm
Let $x\in\widetilde{F}_\eps$. Since $J_\eps$ is radially non-increasing, non-negative and supported in $B_{\eps/2}$, there is some $\widetilde{J}_\eps:\R_+\to\R_+$ such that $J_\eps(z)=\widetilde{J}_\eps(|z|)$ and $\mathrm{supp}(\widetilde{J}_\eps)=[0,\eps/2)$. Thus, passing to polar coordinates, the mass carried by $J_\eps(x-\cdot)$ in $\widetilde{F}_\eps$ can be rewritten as
\begin{align*}
\int_{\widetilde{F}_\eps}J_\eps(x-y)\hspace{0.05em}\mathrm{d}y&=\int_{\Lambda_\eps(x)}J_\eps(y)\hspace{0.05em}\mathrm{d}y=\int_{\S^{N-1}}\left(\int_0^{\eps/2} \mathds{1}_{\Lambda_\eps(x)}(\sigma t)\hspace{0.05em}\widetilde{J}_\eps(t)\hspace{0.05em}t^{N-1}\hspace{0.05em}\mathrm{d}t\right)\mathrm{d}\mathscr{H}^{N-1}(\sigma).
\end{align*}
Notice that $\Lambda_\eps(x)$ is a convex set and that $0\in\Lambda_{\eps}(x)$. In particular, both $t\mapsto\mathds{1}_{\Lambda_\eps(x)}(\sigma t)$ and $t\mapsto\widetilde{J}_\eps(t)$ are non-increasing functions. Hence, using Chebyshev's integral inequality (see e.g. \cite[Theorem 2.5.10, p.40]{Mitrinovic}), we have
\begin{align*}
\int_{\widetilde{F}_\eps}J_\eps(x-y)\hspace{0.05em}\mathrm{d}y&\geq \frac{N}{(\eps/2)^N}\!\int_{\S^{N-1}}\!\left(\int_0^{\eps/2} \mathds{1}_{ \Lambda_\eps(x)}(\sigma t)\hspace{0.05em}t^{N-1}\mathrm{d}t\int_0^{\eps/2}\widetilde{J}_\eps(t)\hspace{0.05em}t^{N-1}\mathrm{d}t\right)\mathrm{d}\mathscr{H}^{N-1}(\sigma),
\end{align*}
Since $J_\eps$ has unit mass and $\mathrm{supp}(J_\eps)=B_{\eps/2}$, one has
\begin{align*}
\int_0^{\eps/2}\widetilde{J}_\eps(t)\hspace{0.05em}t^{N-1}\mathrm{d}t=\sigma_N^{-1}=\big(N|B_1|\big)^{-1},
\end{align*}
where $\sigma_N=\mathscr{H}^{N-1}(\S^{N-1})$. Ergo,
\begin{align}
\int_{\widetilde{F}_\eps}J_\eps(x-y)\hspace{0.05em}\mathrm{d}y&\geq  \frac{1}{|B_{\eps/2}|}\int_{\S^{N-1}}\left(\int_0^{\eps/2} \mathds{1}_{\Lambda_\eps(x)}(\sigma t)\hspace{0.05em}t^{N-1}\hspace{0.05em}\mathrm{d}t\right)\mathrm{d}\mathscr{H}^{N-1}(\sigma) \nonumber\\
&=\frac{1}{|B_{\eps/2}|}\int_{B_{\eps/2}}\mathds{1}_{\Lambda_\eps(x)}(y)\hspace{0.05em}\mathrm{d}y. \nonumber
\end{align}
Since $\widetilde{F}_\eps\subset\R^N\setminus K_\eps$ and $\Lambda_\eps(x)=B_{\eps/2}\cap(\widetilde{F}_\eps-x)$, we get
\begin{align}
\int_{\R^N\setminus K_\eps}J_\eps(x-y)\hspace{0.05em}\mathrm{d}y&\geq \frac{|B_{\eps/2}(x)\cap \widetilde{F}_\eps|}{|B_{\eps/2}|},\, \hbox{ for any }x\in\widetilde{F}_\eps. \label{borne:feps}
\end{align}
Let us now estimate the quantity $|B_{\eps/2}(x)\cap\widetilde{F}_\eps|$. Observe that for $\eps$ small enough, say when $0<\eps< \eps_2:={4}^{-(N-1)}$, one has $\eps/2>2\eps^\gamma$. In particular, this implies that $B_{\eps/2}(x)\cap\widetilde{F}_\eps$ always contains an hyper-rectangle of the form $\mathscr{T}\big((0,\eps/4)\times(0,2\eps^\gamma)\times\cdots\times(0,2\eps^\gamma)\big)$ for some translation $\mathscr{T}$ of $\R^N$, so that
$$|B_{\eps/2}(x)\cap\widetilde{F}_\eps|\geq(\eps/4)\times(2\eps^\gamma)^{N-1}=2^{N-3}\eps^{N+1}.$$
Therefore, recalling \eqref{borne:feps}, we obtain that, for all $0<\eps<\eps_2$ and all $x\in \widetilde{F}_\eps$, it holds
\begin{align}
\int_{\R^N\setminus K_\eps}J_\eps(x-y)\hspace{0.05em}\mathrm{d}y\geq C_1\eps, \label{Pa03}
\end{align}
for some $C_1>0$ depending on $N$ only.
\vskip 0.3cm

\noindent\emph{Step 2: Lower bound in $\R^N\setminus(K_\eps\cup\widetilde{F}_\eps)$}
\vskip 0.3cm
Let us now consider the case where $x\in\R^N\setminus(K_\eps\cup\widetilde{F}_\eps)$. For it, we first note that, since $0<\eps^\gamma<\eps<R_0/2$ (remember $0<\eps< \eps_1$), the point $x_0:=(R_0,\eps^\gamma,0,\cdots,0)\in\partial\widetilde{F}_\eps$ satisfies
$$ |x_0|^2=R_0^2+\eps^{2\gamma}<R_0^2+R_0\eps^\gamma/2<(R_0+\eps^\gamma/4)^2, $$
which implies that $\widetilde{F}_\eps\cap B_{R_0+\eps^\gamma/4}\neq\varnothing$. On the other hand, it is clear from the definition of $\widetilde{F}_\eps$ that $\widetilde{F}_\eps\setminus B_{R_1-\eps^\gamma/4}\neq\varnothing$. A consequence of this is that
$$ \widetilde{F}_\eps\cap\mathcal{A}(R_0+\eps^\gamma/4,R_1-\eps^\gamma/4)=\mathcal{A}(R_0+\eps^\gamma/4,R_1-\eps^\gamma/4)\cap\big\{z\in\R^N;z_1>0,|z'|<\eps^\gamma\big\}. $$
Whence, recalling properties (i) and (iv) in the definition of $K_\eps$, we deduce that
$$\mathcal{A}(R_0+\eps/4,R_1-\eps/4)\subset \mathcal{A}(R_0+\eps^\gamma/4,R_1-\eps^\gamma/4)\subset K_\eps\cup\widetilde{F}_\eps,$$
where, in the left-hand side, we have used the fact that $\eps^\gamma<\eps$.
In turn, this implies that
$$x\in\R^N\setminus\mathcal{A}(R_0+\eps/4,R_1-\eps/4). $$
In particular, since $0<\eps<R_0/2<R_0$, we may find a point $z\in \R^N$ such that
\begin{align}
|x-z|=\frac{3\eps}{8} \hbox{ and } B_{\eps/8}(z)\subset B_{\eps/2}(x)\setminus\overline{\mathcal{A}}\subset\R^N\setminus K_\eps. \label{DE:z}
\end{align}
Indeed, when $x\in\R^N\setminus B_{R_1-\eps/4}$, this follows from the convexity of $B_{R_1}$; and, when $x\in B_{R_0+\eps/4}$, the constraint $0<\eps<R_0/2$ allows one to choose $z$ on the diagonal of $B_{R_0+\eps/4}$ containing $x$. 
On account of this, we may write
\begin{align*}
\int_{\R^N\setminus K_\eps}J_\eps(x-y)\hspace{0.05em}\mathrm{d}y\geq\int_{B_{\eps/8}(z)}J_\eps(x-y)\hspace{0.05em}\mathrm{d}y=\int_{B_{\eps/8}(z-x)}J_\eps(y)\hspace{0.05em}\mathrm{d}y=\int_{B_{1/8}\left(\frac{z-x}{\eps}\right)}J(y)\hspace{0.05em}\mathrm{d}y.
\end{align*}
Now, by \eqref{DE:z}, we have $(z-x)/\eps\in\partial B_{3/8}$. Thus,
\begin{align*}
\int_{\R^N\setminus K_\eps}J_\eps(x-y)\hspace{0.05em}\mathrm{d}y\geq \int_{B_{1/8}(e_x)}J(y)\hspace{0.05em}\mathrm{d}y=:M_{J}(e_x)\, \hbox{ for some }e_x\in \partial B_{3/8}.
\end{align*}
Notice that $B_{1/8}(e_x)\subset B_{1/2}=\mathrm{supp}(J)$ (because $e_x\in\partial B_{3/8}$) which implies $M_{J}(e_x)>0$. Moreover, since $J$ is radially symmetric, the quantity $M_{J}(e_x)$ does not depend on the choice of $e_x\in\partial B_{3/8}$, namely
$$ M_{J}(e_x)=M_{J}(e)\equiv M_{J}>0,\, \hbox{ for every } e\in\partial B_{3/8}, $$
and some constant $M_{J}$ depending on $J$ only.

Therefore, for any $0<\eps<\eps_1$ and $x\in \R^N\setminus(K_\eps\cup\widetilde{F}_\eps)$, it holds
\begin{align}
\int_{\R^N\setminus K_\eps}J_\eps(x-y)\hspace{0.05em}\mathrm{d}y\geq M_{J}>0. \label{Pa}
\end{align}

\noindent\emph{Step 3: Conclusion}
\vskip 0.3cm
Since $\R^N\setminus K_\eps=\widetilde{F}_\eps\cup (\R^N\setminus (\widetilde{F}_\eps\cup K_\eps))$, by \eqref{Pa03} and \eqref{Pa}, we obtain
$$ \inf_{x\in\R^N\setminus K_\eps}\int_{\R^N\setminus K_\eps}J_\eps(x-y)\hspace{0.05em}\mathrm{d}y\geq \min\left\{M_{J},C_1\right\}\eps, $$
for any $0<\eps < \eps_3:=\min\{\eps_1,\eps_2\}$. Whence, letting
$$ \eps_0:=\min\left\{\eps_3,\frac{\min\left\{M_{J},C_1\right\}}{\max_{[0,1]}f'}\right\}, $$
and recalling that $f_\eps(s)=\eps^2f(s)$, we obtain
$$ \max_{[0,1]}f_\eps'<\inf_{x\in\R^N\setminus K_\eps}\int_{\R^N\setminus K_\eps}J_\eps(x-y)\hspace{0.05em}\mathrm{d}x \, \hbox{ for any }\eps\in(0,\eps_0), $$
which is the desired inequality.
\end{proof}

\begin{rem}\label{RK:geo2}
Since $J_\eps$ is radially non-increasing and satisfies \eqref{C01}, there is some non-increasing $\widetilde{J}_\eps\in L_{\mathrm{loc}}^1(0,\infty)$ satisfying $J_\eps(z)=\widetilde{J}_\eps(|z|)$. In particular, since $d_{\mathrm{g}}(x,y)\geq|x-y|$, it holds that
$$ \sup_{x\in\R^N\setminus K_\eps}\int_{\R^N\setminus K_\eps}\widetilde{J}_\eps(d_{\mathrm{g}}(x,y))\hspace{0.05em}\mathrm{d}y\leq \sup_{x\in\R^N\setminus K_\eps}\int_{\R^N\setminus K_\eps}J_\eps(x-y)\,\mathrm{d}y=1, $$
thus implying that $\widetilde{J}_\eps$ satisfies \eqref{J:geo}. Moreover, Proposition \ref{LE:CONTINUE} still holds when $J_\eps(x-y)$ is replaced by $\widetilde{J}_\eps(d_{\mathrm{g}}(x,y))$, i.e. we still have
\begin{align}
\max_{[0,1]}f_\eps'<\inf_{x\in\R^N\setminus K_\eps}\int_{\R^N\setminus K_\eps}\widetilde{J}_\eps(d_{\mathrm{g}}(x,y))\hspace{0.05em}\mathrm{d}y. \label{LE:geo:cont}
\end{align}
Indeed, this is because our proof reduces to estimate the mass carried by $J_\eps(x-\cdot)$ on \emph{convex} sub-domains of $\R^N\setminus K_\eps$ and, in this case, the geodesic distance coincides with the Euclidian distance, namely it holds that $J_\eps(x-y)=\widetilde{J}_\eps(d_{\mathrm{g}}(x,y))$.
\end{rem}

\section{Construction of a global super-solution}\label{section-supersol}

In this section we construct a global super-solution to \eqref{eq-eps}.
Precisely, given a pair $(J,f)$ satisfying \eqref{C02} and \eqref{J-assum}
and given the family of obstacles $(K_\eps)_{0<\eps<1}$ associated to $(J,f)$
(as defined in Section \ref{section-data}), we construct a global super-solution $\bar{u}_\eps$ to
\begin{align}
\int_{\R^N\setminus K_\eps}J_\eps(x-y)(\bar u_\eps(y)-\bar u_\eps(x))\mathrm{d}y +f_\eps(\bar u_\eps(x)) \le 0 \quad\text{for }x \in \R^N\setminus K_\eps, \label{bch-eq-f}
\end{align}
that further satisfies
\begin{align}
\bar u_\eps \equiv 1 \quad\text{for }x \in \R^N\setminus B_R, \label{bch-eq-f-bc}
\end{align}
for some large $R>0$, where $f_\eps$ and $J_\eps$ are as in \eqref{Donnees:fJ}.
More precisely, we prove the following
\begin{lemma}\label{bch-lem-supersol}
Let $N\ge 2$ and let $(J,f)$ be a pair satisfying \eqref{C02} and \eqref{J-assum}. Let $(K_\eps)_{0<\eps<1}$ be the family of obstacles associated to the pair $(J,f)$ (as defined in Section~\ref{section-data}). Let $f_\eps$ and $J_\eps$ be as in \eqref{Donnees:fJ}. Then, there exists $R^*>0$ and $\eps^*>0$ such that, for all $0<\eps <\eps^*$ and all $R\geq R^*$, there is a continuous positive nonconstant function $\bar u_\eps$ satisfying \eqref{bch-eq-f} and \eqref{bch-eq-f-bc}.
\end{lemma}


The proof of Lemma~\ref{bch-lem-supersol} follows essentially two steps. In the first step, we construct a positive solution to a suitable auxiliary problem defined in $B_{R}\setminus K_\eps$ for some large $R$. Then, in a second step, we regularise this solution to obtain  a super-solution  that  satisfies both \eqref{bch-eq-f} and \eqref{bch-eq-f-bc}. To simplify the presentation each step of the proof corresponds to a subsection.

\subsection{An auxiliary problem in $B_R\setminus K_\eps$}\label{SE:AUXPROBLEM:BC}
Let us first construct an adequate auxiliary problem. To do so, we define a new nonlinearity, $\widetilde{f}$, satisfying
\begin{equation}\label{def-tildef}
\tilde f(s):=\left\{\begin{array}{cl}
-\kappa s & \text{for } s\le \frac{3\theta}{4},\vspace{3pt}\\
f_0(s)  & \text{for }  \frac{3\theta}{4}< s< \theta, \vspace{3pt}\\
f(s) & \text{for } \theta\le s\le 1,\vspace{3pt}\\
f'(1)(s-1) & \text{for }  s> 1,
\end{array}
\right.
\end{equation}
where $\theta\in(0,1)$ is as in \eqref{C02}, $\kappa>0$ is a small number and $f_0$ is a smooth function such that $\tilde f \in C^1(\R)$.
From \eqref{C02}, we can choose $\kappa>0$ and $f_0$ such that
\begin{align}
f\le \tilde f\ \text{ in }[0,1],\ \max_{[0,1]}\tilde f(s)=\max_{[0,1]}f\ \text{and}\ \sup_{\R}\tilde f' \le \sup_{[0,1]}f'. \label{Controle:f:auxiliaire}
\end{align}
Now, for $R>R_1+2$, we let $L_{R,\eps}$ be the operator given by
\begin{equation}\label{def-LRE}
L_{R,\eps}w(x):= \int_{B_R\setminus K_\eps}J_\eps(x-y)(w(y)-w(x))\mathrm{d}y,
\end{equation}
and we consider the following problem
\begin{align}
L_{R,\eps}u_{\eps,R}(x)+c_\eps(x)(1-u_{\eps,R}(x))+\tilde f_\eps(u_{\eps,R}(x))=0\ \hbox{ for all }x\in\overline{B_R\setminus K_\eps},\label{KTR1}
\end{align}
where
\begin{align}
\tilde f_\eps (s)=\eps^2\tilde f(s)\text{ for $s\in\R$ and }c_\eps(x):=\int_{\R^N\setminus B_{R}}J_\eps(x-y)\hspace{0.05em}\mathrm{d}y\hbox{ for }x\in B_R\setminus K_\eps. \label{DonneeTildeF}
\end{align}
Our goal in this step is to show that, for each $\eps\in(0,1)$ small enough, there exists a continuous function $u_{\eps,R}:\overline{B_R\setminus K_\eps}\to(0,1)$ satisfying \eqref{KTR1}.
\begin{rem}
Observe that, by construction (remember \eqref{Controle:f:auxiliaire}), the function
$$ \widehat{u}_{\eps,R}:=\left\{\begin{array}{cl} u_{\eps,R} & \text{in }\overline{B_R\setminus K}, \vspace{3pt}\\ 1 & \text{in }\R^N\setminus \overline{B}_R, \end{array}\right. $$
provides a \emph{discontinuous} super-solution to \eqref{bch-eq-f} satisfying \eqref{bch-eq-f-bc}.  We are thus on the right track to construct the required super-solution.
\end{rem}
For it, we observe that, by setting $v_{\eps,R}:=1-u_{\eps,R}$,~\eqref{KTR1} rewrites
\begin{align}
L_{R,\eps}v_{\eps,R}(x)-c_\eps(x)v_{\eps,R}(x)+g_\eps(v_{\eps,R}(x))=0\ \hbox{ for  }x\in\overline{B_R\setminus K_\eps},\label{NEWktr}
\end{align}
with $g_\eps(s):=-\eps^2\tilde{f}(1-s)$. Therefore, to construct $u_{\eps,R}$ it suffices to construct a positive solution $v_{\eps,R}:\overline{B_R\setminus K}\to(0,1)$ to \eqref{NEWktr}.
As in \cite{BHM}, this will be done using a variational argument. To do so,
we define
$$ g(s):=-\tilde{f}(1-s),\ \,  G(t):=\int_0^t g(s)\hspace{0.05em}\mathrm{d}s  \ \text{ and }\ G_\eps(t):=\eps^{2}G(t),$$
for all $s,t\in\R$ and $\eps\in(0,1)$.
Now, for any $\varepsilon\in(0,1)$ and any domain $\Omega\subset B_R\setminus K_\eps$, we consider the following energy functional
\begin{equation}\label{def-enerOmega}
\mathcal{E}_{\eps,\Omega}(w):=\frac{1}{4}\int_{\Omega}\int_{\Omega}J_\eps(x-y)(w(x)-w(y))^2\mathrm{d}x\mathrm{d}y+\frac{1}{2}\int_{\Omega}c_\eps(x)\hspace{0.05em}w^2(x)\hspace{0.05em}\mathrm{d}x-\int_{\Omega}G_\eps(w(x))\hspace{0.05em}\mathrm{d}x,
\end{equation}
for $w\in L^2(\Omega)$.
Observe that for any $\eps>0$ and any domain $\O\subset B_R\setminus K_\eps$, the null function $w\equiv 0$ is a \emph{global} minimiser of $\mathcal{E}_{\eps,\Omega}$. Therefore, we have to construct a \emph{local} minimiser. However, unlike its local analogue, the energy functional $\e_{\eps,\O}$ does \emph{not} posses strong compactness properties, rendering this type of approach very delicate to implement.

With this in mind, we will show that, for the family $K_\eps$ constructed in Section \ref{section-data}  and  $\eps$ small enough, the above energy has indeed a nontrivial local minimiser when $\O=B_R\setminus K_\eps$.

Following the scheme of construction introduced in \cite{BHM}, we first show that the function $w_0:=\mathds{1}_{B_{R_{0}}}$ is a strict minimiser of the functional $\mathcal{E}_{\eps,B_{R_{0}}}$ when $\eps\in(0,1)$ is small enough.

More precisely, 

\begin{prop}\label{Poincare}
Let $N\ge 2$, $0<R_0<R_0^*(J,f)$ (where $R_0^*(J,f)$ is given by \eqref{R0}) and let $w_0:=\mathds{1}_{B_{R_{0}}}$. Then, there  exists $\kappa_0>0$, $0<\eps_1(J,N,R_0)<1$ and $0<\delta_0(R_0)<|B_{R_{0}}|^{1/2}$ such that, for each $0<\eps< \eps_{1}$, it holds that
$$\mathcal{E}_{\eps,B_{R_{0}}}(w)-\mathcal{E}_{\eps,B_{R_{0}}}(w_0)\ge \kappa_0 \eps^2\|w-w_0\|_{L^2(B_{R_{0}})}^2,$$
for all $w\in L^2(B_{R_0})$ such that $\|w-w_0\|_{L^2(B_{R_{0}})}\le \delta_0$.
\end{prop}

\begin{proof}
Let us begin with some preliminary observations. First, we notice that since
$g$ is linear around $1$ (because $\widetilde{f}$ is linear around $0$), the
function $G_\eps$ is smooth in a neighborhood of $1$. In particular, there
exists $\tau_0(\widetilde{f})>0$ such that
$$ G_\eps(t)=G_\eps(1)+G_\eps'(1)(t-1)+\frac{1}{2}\hspace{0.05em}G_\eps''(1)(t-1)^2 \quad\text{for any }|t-1|<\tau_0. $$
But since $G_\eps'(1)=\eps^2G'(1)=\eps^2g(1)=0$ and $G_\eps''(1)=\eps^2G''(1)=\eps^2g'(1)=-\eps^2\widetilde{f}'(0)=-\eps^2\kappa$,
this expansion can be rewritten as
\begin{align}
G_\eps(t)=\eps^2G(1)-\frac{\kappa\hspace{0.05em}\eps^2}{2}(t-1)^2\quad\text{for any }|t-1|<\tau_0. \label{tau0}
\end{align}
Using the number $\tau_0$, we define
\begin{align}
\delta_0:=\min\left\{\frac{\theta}{4},\frac{C_0}{\kappa},\frac{\tau_0}{2}\right\}|B_{R_0}|^{1/2}, \label{Construction:de:delta0}
\end{align}
where $\theta$, $C_0$ and $\kappa$ are as in \eqref{C02}, \eqref{C_0} and \eqref{def-tildef}; and we let $w\in L^2(B_{R_0})$ be such that
\begin{align}
\|w-w_0\|_{L^2(B_{R_0})}\leq \delta_0. \label{Hypothese:sur:w:delta}
\end{align}
Second, denoting by $\langle w_0\rangle:=\mathrm{span}_{L^2(B_{R_{0}})}(w_0)$ the vector space spanned by $w_0$ and letting
$\langle w_0\rangle^\perp$ be its orthogonal with respect to the standard scalar product of $L^2(B_{R_{0}})$, we can
write the space $L^2(B_R)$ as the direct sum $L^{2}(B_{R_{0}})= \langle w_0\rangle \oplus \langle w_0\rangle^\perp$.
This means that we may always find a constant $\alpha\in\R$ and a
function $h\in\langle w_0\rangle^\perp$ such that $w$ decomposes as $w=\alpha\hspace{0.05em}w_0+h$.
In particular, the orthogonality of $h$ with respect to $w_0$ implies that
\begin{align}
\int_{B_{R_{0}}}h(x)\hspace{0.05em}\mathrm{d}x=0~~\text{and}~~ \|w-w_0\|_{L^2(B_{R_{0}})}^2=(1-\alpha)^2\|w_0\|_{L^2(B_{R_{0}})}^2+\|h\|_{L^2(B_{R_{0}})}^2. \label{est:h:alpha}
\end{align}
In view of this, assumption \eqref{Hypothese:sur:w:delta} gives
\begin{align}
-\frac{\delta_0}{|B_{R_{0}}|^{1/2}}\leq (1-\alpha)\leq \frac{\delta_0}{|B_{R_{0}}|^{1/2}}  \ \text{ and }\ \|h\|_{L^2(B_{R_{0}})}\le \delta_0 . \label{EST:alpha}
\end{align}
This fact will be abundantly used in the sequel.

This being said, we are now in position to prove Proposition~\ref{Poincare}.
For it, we observe that, since $w_0\equiv1$ in $B_{R_0}$, we have that
$$\mathcal{E}_{\eps,B_{R_{0}}}(w_0)=-\int_{B_{R_{0}}}G_\eps(w_0(x))\hspace{0.05em}\mathrm{d}x=-G_\eps(1)|B_{R_{0}}|=-\eps^2\,G(1)|B_{R_{0}}|. $$
Furthermore, thanks to $R>R_0+2$ and $\mathrm{supp}(J_\eps)\subset B_{\frac{\eps}{2}}$, we have that $c_\eps(x)\equiv 0$ in $B_{R_{0}}$, for any $0<\eps<1$. Consequently, $\e_{\eps,B_{R_{0}}}(w)$ rewrites
$$
\begin{array}{lccc}
\e_{\eps,B_{R_{0}}}(w)&=& \underbrace{\frac{1}{4}\int_{B_{R_{0}}}\int_{B_{R_{0}}} J_\eps(x-y)(w(x)-w(y))^2\mathrm{d}x\mathrm{d}y} &\underbrace{-\int_{B_{R_{0}}}G_\eps(w(x))\hspace{0.05em}\mathrm{d}x} .\\
& & \text{\emph{II}} & I
\end{array}
$$
Let us first estimate $\text{\emph{II}}$. In view of the Bourgain-Brezis-Mironescu representation of $H^1(B_{R_0})$ (see \cite{BBM}), one can interpret $\text{\emph{II}}$ as a nonlocal approximation of $\|\nabla w\|_{L^2(B_{R_0})}^2$. The crux of our strategy is that, as shown by Ponce \cite[Theorem 1.1]{Ponce2004}, this nonlocal approximation enjoys a Poincar\'e-type inequality. Let us now proceed.
Let $(\rho_\eps)_{0<\eps<1}$ be the family of radially symmetric mollifiers defined by
$$\rho_{\eps}(z):=M_2(J)^{-1}J_{\eps}(z)|z|^2\hspace{0.1em}\eps^{-2}\,\hbox{ for }\eps\in(0,1),$$
where $M_2(J)$ is given by \eqref{M2}.
Notice that, by construction, it satisfies
\begin{align*}
\rho_\varepsilon\geq0\ \hbox{ a.e. in}~~\mathbb{R}^N,~~\displaystyle\int_{\mathbb{R}^N}\rho_\varepsilon(z)\hspace{0.05em}\mathrm{d}z=1~~\text{and}~~\lim_{\varepsilon\to0^+}\int_{|z|\geq\tau}\rho_\varepsilon(z)\hspace{0.05em}\mathrm{d}z=0,
\end{align*}
for each $0<\eps<1$ and each $\tau>0$. Moreover, $\text{\emph{II}}$ can be rewritten as
$$\text{\emph{II}}=\eps^2\,\frac{M_2(J)}{4}\int_{B_{R_0}}\int_{B_{R_0}}\rho_\eps(x-y)\frac{|w(x)-w(y)|^2}{|x-y|^2}\mathrm{d}x\mathrm{d}y.$$
Now, by \cite[Theorem 1.1]{Ponce2004}, we know that there exists some $\eps_{1}=\eps_1(J,N,R_0)>0$ such that the following Poincar\'e-type inequality
$$\left\|w-\fint_{B_{R_0}} w\right\|_{L^2(B_{R_0})}^2\leq \frac{2\hspace{0.05em}A_0}{K_{2,N}} \int_{B_{R_0}}\int_{B_{R_0}}\rho_\eps(x-y)\frac{|w(x)-w(y)|^2}{|x-y|^2}\mathrm{d}x\mathrm{d}y\left(=\frac{8\hspace{0.05em}A_0\hspace{0.1em}\eps^{-2}}{K_{2,N}\hspace{0.1em}M_2(J)}\times\text{\emph{II}}\right), $$
holds for all $\eps\in(0,\eps_{1})$ and all $w\in L^2(B_{R_0})$. Here,
$$ K_{2,N}:=\int_{\S^{N-1}}(\sigma\cdot e_1)^2\,\mathrm{d}\mathcal{H}^{N-1}(\sigma)=\frac{1}{N}, $$
and $A_0>0$ is the smallest constant such that the standard Poincar\'e-Wirtinger inequality holds. That is, $A_0$ is the smallest positive constant such that
$$\left\|w-\fint_{B_{R_0}} w\right\|_{L^2(B_{R_0})}^2\!\leq A_0\|\nabla w\|_{L^2(B_{R_0})}^2, $$
holds for any $w\in H^1(B_{R_0})$. In our case, $A_0$ satisfies the upper bound:
$$A_0\leq\frac{\mathrm{diam}(B_{R_0})^2}{\pi^2}=\frac{4R_0^2}{\pi^2}, $$
see \cite[Theorem 3.2]{Beben} (see also \cite{Payne1960}). In particular, this gives
$$ \eps^2\,\frac{\pi^2\hspace{0.05em}M_2(J)}{32\hspace{0.05em}N\hspace{0.05em}R_0^2}\left\|w-\fint_{B_{R_0}} w\right\|_{L^2(B_{R_0})}^2\!\leq \text{\emph{II}}. $$
Now, since $w=\alpha\hspace{0.05em}w_0 +h$, since $w_0\equiv1$ on $B_{R_0}$ and since $h$ is integral free (by \eqref{est:h:alpha}) we have 	
\begin{align}
\text{\emph{II}}\geq\eps^2\,\frac{C_{N,J}}{R_0^2}\|h\|_{L^2(B_{R_0})}^2, \label{contribution:J}
\end{align}
where $C_{N,J}$ is given by \eqref{CNJ}.
We are now left to estimate $I$.  For it, we rewrite $I$ as follows
\begin{equation}\label{I}
 I=\e_{\eps,B_{R_{0}}}(w_0)-\int_{B_{R_{0}}}\big[G_\eps(w(x))-G_\eps(w_0(x))\big]\mathrm{d}x.
 \end{equation}
To estimate the last integral, we split it into two parts, $I_1$ and $I_2$, where
\begin{align*}
I_1&:=-\!\int_{B_{R_{0}}}\!\big[G_\eps(w_0\!+\!(\alpha\!-\!1) w_0\!+\!h)\!-\!G_\eps(w_0\!+\!(\alpha\!-\!1)w_0)\big], \\
I_2&:=-\!\int_{B_{R_{0}}}\!\big[G_\eps(w_0\!+\!(\alpha\!-\!1)w_0)\!-\!G_\eps(w_0)\big].
\end{align*}
Let us first estimate $I_2$. Using \eqref{Construction:de:delta0}, \eqref{Hypothese:sur:w:delta} and \eqref{EST:alpha} we have in particular that $|1-\alpha|<\tau_0$. This, together with \eqref{tau0}, gives
\begin{align*}
I_2=-\int_{B_{R_{0}}}\big[G_\eps(w_0\!+\!(\alpha\!-\!1)w_0)\!-\!G_\eps(w_0)\big]= \frac{\kappa}{2}\,\eps^2|B_{R_{0}}|(\alpha-1)^2.
\end{align*}
Therefore, recalling \eqref{I}, we get
\begin{equation}
I= \e_{\eps,B_{R_{0}}}(w_0) +\frac{\kappa}{2}\,\eps^2|B_{R_{0}}|(\alpha-1)^2 +I_1. \label{I2}
\end{equation}
Let us now estimate $I_1$. On account of \eqref{EST:alpha}, we may write
\begin{align}
\alpha=1-\eta\quad\text{for some }|\eta|\leq\frac{\delta_0}{|B_{R_0}|^{1/2}}. \label{parametre:eta}
\end{align}
Then, a standard change of variables yields
\begin{align*}
I_1=-\!\int_{B_{R_{0}}}\!\int_{\alpha}^{\alpha+h(x)}\!\!g_\eps(\tau)\mathrm{d}\tau\mathrm{d}x&=\eps^2\!\int_{B_{R_{0}}}\!\int_{1-\eta}^{1-\eta+h(x)}\!\widetilde{f}(1\!-\!\tau)\mathrm{d}\tau\mathrm{d}x=-\eps^2\!\int_{B_{R_{0}}}\!\int_{0}^{-h(x)}\!\widetilde{f}(\tau\!+\!\eta)\mathrm{d}\tau\mathrm{d}x.
\end{align*}
Now, we set
$$\Sigma:=\left\{x\in B_{R_{0}};-h(x)>\frac{\theta}{2}\right\},$$
and we decompose $I_1$ as
\begin{align}
I_1=-\eps^2\left(\int_{\Sigma}\int_{0}^{-h(x)}\!\widetilde{f}(\tau+\eta)\hspace{0.05em}\mathrm{d}\tau\mathrm{d}x +\int_{B_{R_{0}}\setminus\Sigma}\int_{0}^{-h(x)}\!\widetilde{f}(\tau+\eta)\hspace{0.05em}\mathrm{d}\tau\mathrm{d}x \right). \label{Decomposition:de:Iun}
\end{align}
We will estimate these two integrals separately. In view of \eqref{Construction:de:delta0} and \eqref{parametre:eta}, we have that $|\eta|\leq \theta/4$. In turn, this implies that
$$ -h(x)+|\eta|\leq\frac{3\theta}{4}\quad\text{for any }x\in B_{R_0}\setminus \Sigma. $$
Since, by construction, $\widetilde{f}$ is linear in $(-\infty,3\theta/4]$, we get
\begin{align*}
\int_{B_{R_{0}}\setminus\Sigma}\int_{0}^{-h(x)}\!\widetilde{f}(\tau+\eta)\hspace{0.05em} \mathrm{d}\tau\mathrm{d}x&=-\kappa\int_{B_{R_{0}}\setminus \Sigma}\left(\int_{0}^{\eta-h(x)}\tau\hspace{0.05em}\mathrm{d}\tau - \int_{0}^{\eta} \tau\hspace{0.05em} \mathrm{d}\tau\right)\mathrm{d}x \\
&=  -\frac{\kappa}{2} \int_{B_{R_{0}}\setminus\Sigma}h^2(x)\hspace{0.05em} \mathrm{d}x +\kappa\eta \int_{B_{R_{0}}\setminus \Sigma	}h(x)\hspace{0.05em}\mathrm{d}x
\\
&=  -\frac{\kappa}{2} \int_{B_{R_{0}}\setminus\Sigma}h^2(x)\hspace{0.05em}  \mathrm{d}x -\kappa\eta \int_{\Sigma}h(x)\hspace{0.05em}\mathrm{d}x,
\end{align*}
where, in the last equality, we have used the fact that $h$ is integral free, that is:
$$\int_{B_{R_{0}}\setminus \Sigma	}h(x)\hspace{0.05em}\mathrm{d}x+\int_{\Sigma}h(x)\hspace{0.05em}\mathrm{d}x=0.$$
Using now the Cauchy-Schwarz inequality, we get
\begin{equation*}
\int_{B_{R_{0}}\setminus\Sigma}\int_{0}^{-h(x)}\!\widetilde{f}(\tau+\eta)\hspace{0.05em}\mathrm{d}\tau\mathrm{d}x\le -\frac{\kappa}{2} \int_{B_{R_{0}}\setminus\Sigma}h^2(x)\hspace{0.05em} \mathrm{d}x +\kappa|\eta|\sqrt{|\Sigma|} \|h\|_{L^2(\Sigma)}.
\end{equation*}
By the Bienaym\'e-Chebyshev inequality, we have
\begin{align}
|\Sigma|\leq\left(\frac{2}{\theta}\right)^2\|h\|^2_{L^2(\Sigma)}, \label{Sigma:Chebyshev}
\end{align}
and thus
\begin{equation}
\label{esti-I1-1-1}
\int_{B_{R_{0}}\setminus\Sigma}\int_{0}^{-h(x)}\!\widetilde{f}(\tau+\eta)\hspace{0.05em} \mathrm{d}\tau\mathrm{d}x\le -\frac{\kappa}{2} \int_{B_{R_{0}}\setminus\Sigma}h^2(x)\hspace{0.05em} \mathrm{d}x +\frac{2\kappa|\eta|}{\theta} \|h\|^2_{L^2(\Sigma)}.
\end{equation}
Thanks to \eqref{Construction:de:delta0} and \eqref{parametre:eta}, \eqref{esti-I1-1-1} reduces to
\begin{equation}
\label{esti-I1-1}
\int_{B_{R_{0}}\setminus\Sigma}\int_{0}^{-h(x)}\!\widetilde{f}(\tau+\eta)\hspace{0.05em} \mathrm{d}\tau\mathrm{d}x\le -\frac{\kappa}{2} \int_{B_{R_{0}}\setminus\Sigma}h^2(x)\hspace{0.05em} \mathrm{d}x +\frac{2C_0}{\theta} \|h\|^2_{L^2(\Sigma)}.
\end{equation}
Let us now estimate the first integral on the right-hand side of \eqref{Decomposition:de:Iun}. For it, we observe that
$$ \tau+\eta\geq-|\eta|\geq\frac{-\delta_0}{\sqrt{|B_{R_0}|}}\geq-\frac{C_0}{\kappa}\quad\text{for any }x\in\Sigma\text{ and any }\tau\in(0,-h(x)). $$
Recalling \eqref{Controle:f:auxiliaire}, we then obtain
$$ \sup_{x\in\Sigma}\sup_{\tau\in(0,-h(x))}\widetilde{f}(\tau+\eta)\leq \sup_{s\geq-\frac{C_0}{\kappa}}\widetilde{f}(s)=\max_{s\in[0,1]}\widetilde{f}(s)=C_0. $$
This, together with the Cauchy-Schwarz inequality, gives
\begin{equation*}
\int_{\Sigma}\int_{0}^{-h(x)}\!\widetilde{f}(\tau+\eta)\hspace{0.05em} \mathrm{d}\tau\mathrm{d}x \le C_0 \int_{\Sigma}|h(x)|\hspace{0.05em}\mathrm{d}x\le C_0\sqrt{|\Sigma|}\|h\|_{L^2(\Sigma)},
\end{equation*}
Using the Bienaym\'e-Chebyshev inequality \eqref{Sigma:Chebyshev}, we finally get
\begin{equation}\label{esti-I1-2}
\int_{\Sigma}\int_{0}^{-h(x)}\!\widetilde{f}(\tau+\eta)\hspace{0.05em}\mathrm{d}\tau\mathrm{d}x \le \frac{2C_0}{\theta}\|h\|_{L^2(\Sigma)}^2.
\end{equation}
Collecting \eqref{contribution:J}, \eqref{I2}, \eqref{esti-I1-1} and \eqref{esti-I1-2}, we obtain that
\begin{align*}
\e_{\eps,B_{R_{0}}}(w)&-\e_{\eps,B_{R_{0}}}(w_0)\\
&\ge \eps^2\left( \frac{\kappa}{2}\,|B_{R_{0}}|(\alpha-1)^2+ \frac{\kappa}{2} \|h\|^2_{L^2(B_{R_{0}}\setminus \Sigma)} -\frac{4C_0}{\theta}\|h\|_{L^2(\Sigma)}^2+\,\frac{C_{N,J}}{R_0^2}\|h\|_{L^2(B_{R_{0}})}^2\right),
\end{align*}
for all $0<\eps<\eps_{1}$ and all $w\in L^2(B_{R_0})$ with $\|w-w_0\|_{L^2(B_{R_{0}})}\le \delta_0$.
Recalling that $0<R_0\le R^*_0(J,f)$ and using \eqref{R0}, we have $C_{N,J}/R_0^2 \ge 5C_0/\theta$. This, together with the above inequality, yields
 \begin{align*}
\e_{\eps,B_{R_{0}}}(w)-\e_{\eps,B_{R_{0}}}(w_0)&\ge \eps^2\left( \frac{\kappa}{2}\,|B_{R_{0}}|(\alpha-1)^2+ \frac{C_0}{\theta}\|h\|_{L^2(B_{R_{0}})}^2\right).
\end{align*}
Therefore, letting
$$ \kappa_0:=\inf\left\{\frac{\kappa}{2},\frac{C_0}{\theta}\right\}, $$
and recalling \eqref{est:h:alpha}, we obtain
\begin{equation*}
\e_{\eps,B_{R_{0}}}(w)-\e_{\eps,B_{R_{0}}}(w_0)\ge   \eps^2\kappa_0 \|w-w_0\|^2_{L^2(B_{R_{0}})},
\end{equation*}
for all $0<\eps<\eps_{1}$ and all $w\in L^2(B_{R_0})$ with $\|w-w_0\|_{L^2(B_{R_{0}})}\le \delta_0$.
\end{proof}
\begin{rem}\label{nonlocaltolocal}
Note that the proof of Proposition~\ref{Poincare} relies only on elementary $L^2$-estimates and on a Poincar\'e-type inequality. Remarkably, this allows to adapt straightforwardly our arguments to the local analogue of $\e_{\eps,B_{R_0}}$.
\end{rem}
Using Proposition~\ref{Poincare}, we now prove the following
\begin{prop}\label{Poincaree}
Let $N\ge 2$, and let $\e_{\eps,R}$ be the energy functional defined by \eqref{def-enerOmega} with $\O=B_R\setminus K_\eps$. Then, there  exists $C^*>0$, $0<\delta_0<|B_{R_{0}}|^{1/2}$ and $0<\eps_{\delta_0}<1$ such that, for any $0<\eps < \eps_{\delta_0}$ and any $w\in L^2(B_R\setminus K_\eps)$ with $\|w-w_0\|_{L^2(B_R\setminus K_\eps)}=\delta_0$, it holds that
$$ \mathcal{E}_{\eps,R}(w)-\mathcal{E}_{\eps,R}(w_0)>C^*\eps^2.$$
\end{prop}
\begin{proof}
Let us first notice that our assumptions on $\tilde{f}$ imply that there is some $\kappa_1>0$ such that
\begin{align}
-G(t)\geq\kappa_1\hspace{0.05em}t^2 \, \mbox{ for every }t\in\R. \label{kappa}
\end{align}
Let us now compute the energy of $w_0$. Since $\mathrm{supp}(J_\eps)=B_{\eps/2}$ and $R_{1}-R_0>\eps$, a straightforward calculation yields
\begin{align}
\mathcal{E}_{\eps,R}(w_0)&=\mathcal{E}_{\eps,B_{R_{0}}}(w_0)\!+\!\frac{1}{2}\int_{F_\eps}\int_{B_{R_{0}}}J_\eps(x-y)\hspace{0.05em}\mathrm{d}x\mathrm{d}y. \nonumber
\end{align}
In addition, elementary computations yield
$$ \frac{1}{2}\int_{F_\eps}\int_{B_{R_{0}}}J_\eps(x-y)\hspace{0.05em}\mathrm{d}x\mathrm{d}y=\frac{1}{2}\int_{F_\eps\cap B_{R_0+\frac{\eps}{2}}}\left(\int_{B_{R_{0}}}J_\eps(x-y)\hspace{0.05em}\mathrm{d}x\right)\mathrm{d}y\leq \frac{|F_\eps\cap B_{R_0+\frac{\eps}{2}}|}{2}\leq C\eps^{N+1}, $$
for come constant $C=C(N)>0$. As a consequence, we obtain
\begin{align}
\mathcal{E}_{\eps,R}(w_0)\leq \mathcal{E}_{\eps,B_{R_{0}}}(w_0)+C\eps^{N+1}. \label{GENE:w0}
\end{align}
Next, developing $\mathcal{E}_{\eps,R}(w)$, we get
\begin{align}
\mathcal{E}_{\eps,R}(w)&=\mathcal{E}_{\eps,B_{R_{0}}}(w)+\mathcal{E}_{\eps,F_\eps}(w)+\mathcal{E}_{\eps,B_R \setminus B_{R_1}}(w) \nonumber \\
&\qquad+\frac{1}{2}\int_{F_\eps}\left(\int_{B_{R_{0}}}+\int_{B_R \setminus B_{R_1}}\right)J_\eps(x-y)(w(x)-w(y))^2\mathrm{d}x\mathrm{d}y. \nonumber
\end{align}
Using \eqref{kappa} we obtain that $\mathcal{E}_{\eps,\Omega}(w)\geq \kappa_1\hspace{0.05em}\eps^2\|w\|_{L^2(\Omega)}^2$ for any domain $\Omega\subset B_R\setminus K_\eps$. In particular, since $w_0=0$ in $F_\eps\cup B_R \setminus B_{R_1}$ we have
\begin{align}
\mathcal{E}_{\eps,B_R\setminus K_\eps}(w)\geq \mathcal{E}_{\eps,B_{R_{0}}}(w)+\kappa_1\hspace{0.05em}\eps^2\|w-w_0\|_{L^2(F_\eps\cup B_R \setminus B_{R_1})}^2. \label{GENE:w}
\end{align}
Gluing together \eqref{GENE:w0} and \eqref{GENE:w}, we obtain
\begin{align}
\mathcal{E}_{\eps,R}(w)-\mathcal{E}_{\eps,R}(w_0)\geq \mathcal{E}_{\eps,B_{R_{0}}}(w)-\mathcal{E}_{\eps,B_{R_{0}}}(w_0)+\kappa_1 \hspace{0.05em}\eps^2\|w-w_0\|_{L^2(F_\eps\cup B_R \setminus B_{R_1})}^2-C\eps^{N+1}. \label{1ereEST}
\end{align}
Now, by Proposition \ref{Poincare}, there exists $\kappa_0>0$, $0<\delta_0<|B_{R_{0}}|^{1/2}$ and $\eps_{1}>0$ such that, for any $0<\eps<\eps_{1}$ and any $w\in L^2(B_{R_0})$ with $\|w-w_0\|_{L^2(B_{R_{0}})}\le\delta_0$, we have
\begin{equation}\label{bch-eq-d0}
\mathcal{E}_{\eps,B_{R_{0}}}(w)-\mathcal{E}_{\eps,B_{R_{0}}}(w_0)\ge \kappa_0\hspace{0.05em}\eps^2\|w-w_0\|_{L^2(B_{R_{0}})}^2.
\end{equation}
Letting $\bar \kappa :=\min\{\kappa_1,\kappa_0\}$ and combining \eqref{bch-eq-d0} and \eqref{1ereEST}, we obtain
\begin{align}
\mathcal{E}_{\eps,R}(w)-\mathcal{E}_{\eps,R}(w_0)&\geq \eps^2\bar\kappa\|w-w_0\|_{L^2(B_R\setminus K_\eps)}^2 -C\eps^{N+1}=\eps^2\left(\bar\kappa \delta_0^2-C\eps^{N-1}\right),\label{2emeEST}
\end{align}
for all $0<\eps< \eps_{1}$ and all $w\in L^2(B_R\setminus K_\eps)$ with $\|w-w_0\|_{L^2(B_R\setminus K_\eps)}=\delta_0$. 
The conclusion now follows from \eqref{2emeEST} and the choice
$$ C^*=\frac{\bar \kappa \delta_0^2}{2}\quad\text{and}\quad\eps_{\delta_0}:=\min\left\{\eps_{1}, \left(\frac{\bar \kappa \delta_0^2}{2C}\right)^{\frac{1}{N-1}} \right\}. $$
The proof is thereby complete.
\end{proof}

We are now in position to construct a positive solution to \eqref{NEWktr}. 
\begin{prop}\label{nmeClm}
Let $N\ge 2$ and let $(J,f)$ be a pair satisfying \eqref{C02} and \eqref{J-assum}. Let $(K_\eps)_{0<\eps<1}$ be the family of obstacles associated to the pair $(J,f)$ (as defined in Section~\ref{section-data}).
Let $\tilde{f}$ be the extension of $f$ given by \eqref{def-tildef} and let $\tilde{f}_\eps$ and $J_\eps$ be respectively given by \eqref{DonneeTildeF} and \eqref{Donnees:fJ}.
Then, there exists  $\bar \eps>0$ such that, for all $0<\eps<\bar \eps$, there is a function $v_{\eps,R}\in C(\overline{B_R\setminus K_\eps})$ satisfying~\eqref{NEWktr} and $0<v_{\eps,R}<1$ in $\overline{B_R\setminus K_\eps}$.
\end{prop}

\begin{proof}
Let $w_0:=\mathds{1}_{B_{R_0}}$ and let $0<\delta_0<|B_{R_{0}}|^{1/2}$ and $0<\eps_{\delta_0}<1$ be quantities constructed in the proof of Proposition~\ref{Poincaree}, namely such that
$$ \mathcal{E}_{\eps,R}(w)-\mathcal{E}_{\eps,R}(w_0)>C^*\eps^2,$$
holds for some constant $C^*>0$ and for any $0<\eps<\eps_{\delta_0}$ and any $w\in L^2(B_R\setminus K_\eps)$ with $\|w-w_0\|_{L^2(B_R\setminus K_\eps)}=\delta_0$. Let us fix $0<\eps<\bar{\eps}:=\min\{\eps_0,\eps_{\delta_0}\}$ where $\eps_0$ is as in Proposition~\ref{LE:CONTINUE}. Further, we denote by $\mathbb{B}_{\delta_0}(w_0)$ the following set:
$$ \mathbb{B}_{\delta_0}(w_0):=\big\{w\in L^2(B_R\setminus K_\eps)\,;\ \|w-w_0\|_{L^2(B_R\setminus K_\eps)}\leq\delta_0\big\}, $$
and we define
$$m:=\inf_{w\in \mathbb{B}_{\delta_0}(w_0)} \e_{\eps,R}(w).$$
Note that $m$ is well-defined since $\e_{\eps,R}$ is a non-negative continuous functional in $L^2(B_R\setminus K_\eps)$.

Using Lemma~\ref{Poincaree}, we will show that there is a \emph{local minimum} $v_{\eps,R}$ of the energy $\mathcal{E}_{\eps,R}$ in the ball $\mathbb{B}_{\delta_0}(w_0)$ which is also a solution to \eqref{NEWktr}. However, it must be noted that $\mathcal{E}_{\eps,R}$ lacks of strong compactness properties and passing to the limit along a subsequence is not straightforward. So let us first show that $m$ is achieved in $\mathbb{B}_{\delta_0}(w_0)$.

Take a minimising sequence $(v_j)_{j\in\N}\subset \mathbb{B}_{\delta_0}(w_0)$.
Notice that $|w|\in \mathbb{B}_{\delta_0}(w_0)$ for all $w\in \mathbb{B}_{\delta_0}(w_0)$.
Moreover, a straightforward computation shows that $\mathcal{E}_{\eps,R}(|v_j|)\leq \mathcal{E}_{\eps,R}(v_j)$ for all $j\geq0$. Thus, we may assume that the $v_j$'s are a.e. non-negative for every $j\geq0$. By \eqref{kappa}, we have $-G_\eps(t)\geq \kappa_1\hspace{0.05em}\eps^2\hspace{0.05em}t^2$ for all $t\in\R$. In particular, $\mathcal{E}_{\eps,R}(v_j)\geq \kappa_1\hspace{0.05em}\eps^2\hspace{0.05em}\|v_j\|_{L^2(B_R\setminus K_\eps)}^2$ for all $j\geq0$. Therefore $(v_j)_{j\in\N}$ is bounded in $L^2(B_R\setminus K_\eps)$. Whence, up to extract a subsequence, we obtain that $v_j$ converges weakly in $L^2(B_R\setminus K_\eps)$ to some $v_{\eps,R}\in \mathbb{B}_{\delta_0}(w_0)$ (notice that $\mathbb{B}_{\delta_0}(w_0)$ is \emph{closed} in $L^2(B_R\setminus K_\eps)$). Let us check that $v_{\eps,R}$ is indeed a minimiser of $\mathcal{E}_{\eps,R}$ in $\mathbb{B}_{\delta_0}(w_0)$. To this end, we shall introduce the following notations
$$ \mathcal{J}_\eps(x):=\int_{\R^N\setminus K_\eps} J_\eps(x-y)\hspace{0.05em} \mathrm{d}y\ \text{ and }\ H_\eps(x,s):=\int_0^{s}\big(\mathcal{J}_\eps(x)\hspace{0.1em}\tau -g_\eps(\tau)\big)\hspace{0.05em} \mathrm{d}\tau. $$
Since $0<\eps<\eps_0$, by Proposition \ref{LE:CONTINUE}, we have
\begin{align*}
\max_{[0,1]}f_\eps'<\inf_{\R^N\setminus K_\eps } \mathcal{J}_\eps. 
\end{align*}
Therefore, from the construction of $g_\eps$ (remember \eqref{Controle:f:auxiliaire}), we have
\begin{align}
g'_\eps(s)=\eps^2\tilde f'(1-s)\le \max_{\R}\widetilde{f}_\eps'\le \max_{[0,1]}f_\eps'<\inf_{\R^N\setminus K_\eps } \mathcal{J}_\eps \quad\text{for any }s\in\R. \label{Cdn:assurant:continuite}
\end{align}
Whence, $H_\eps(x,\cdot)$ is convex for each fixed $x$.
Developing the terms involved in the definition of $\mathcal{E}_{\eps,R}$ we arrive at
$$ \mathcal{E}_{\eps,R}(w)=-\frac{1}{2}\int_{B_R\setminus K_\eps}\int_{B_R\setminus K_\eps}J_\eps(x-y)\hspace{0.05em}w(x)\hspace{0.05em}w(y)\hspace{0.05em}\mathrm{d}x\mathrm{d}y+\int_{B_R\setminus K_\eps} H_\eps(x,w(x))\hspace{0.05em}\mathrm{d}x. $$
Using the weak convergence of $(v_j)_{j\in\N}$ towards $v_{\eps,R}$ and the dominated convergence theorem, we can pass to the limit in the double integral and  get that
$$\lim_{j\to+\infty}\int_{B_R\setminus K_\eps}\int_{B_R\setminus K_\eps}J_\eps(x-y)\hspace{0.05em}v_j(x)\hspace{0.05em}v_j(y)\hspace{0.05em}\mathrm{d}x\mathrm{d}y= \int_{B_R\setminus K_\eps}\int_{B_R\setminus K_\eps}J_\eps(x-y)\hspace{0.05em}v_{\eps,R}(x)\hspace{0.05em}v_{\eps,R}(y)\hspace{0.05em}\mathrm{d}x\mathrm{d}y.$$
Moreover, since $H_\eps(x,\cdot)$ is convex, we have
$$\int_{B_R\setminus K_\eps} \big[H_\eps(x,v_j(x))-H_\eps(x,v_{\eps,R}(x))\big]\mathrm{d}x\geq \int_{B_R\setminus K_\eps} \partial_sH_\eps(x,v_{\eps,R}(x))\hspace{0.05em}(v_j(x)-v_{\eps,R}(x))\mathrm{d}x.$$
From the definition of $H_\eps$, $g_\eps$ and from \eqref{Cdn:assurant:continuite} a quick computation shows that $|\partial_sH_\eps(x,s)|=|\mathcal{J}_\eps(x)s-g_\eps(s)|\leq A\hspace{0.05em}|s|$ for all $s\in\R$
and some constant $A>0$. Since $v_{\eps,R}\in L^2(B_R\setminus K_\eps)$, it follows that $\partial_sH_\eps(\cdot,v_{\eps,R}(\cdot))\in L^2(B_R\setminus K_\eps)$. Therefore, using the previous two displayed formulas and the weak convergence of $v_j$ towards $v_{\eps,R}$, we obtain $\lim_{j\to\infty}[\mathcal{E}_{\eps,R}(v_j)-\mathcal{E}_{\eps,R}(v_{\eps,R})]\geq 0$. Since, on the other hand, $\lim_{j\to\infty}\,\mathcal{E}_{\eps,R}(v_j)=m\leq \mathcal{E}_{\eps,R}(v_{\eps,R})$, we finally obtain
$$ \mathcal{E}_{\eps,R}(v_{\eps,R})=m=\inf_{w\in \mathbb{B}_{\delta_0}(w_0)}\,\mathcal{E}_{\eps,R}(w)\le \mathcal{E}_{\eps,R}(w_0). $$
Now, thanks to Proposition ~\ref{Poincaree}, we deduce that $v_{\eps,R}\in \mathbb{B}_{\delta_0}(w_0) $ is a local minimiser and, as such,  $v_{\eps,R}$ solves~\eqref{NEWktr} almost everywhere in $\overline{B_R\setminus K_\eps}$.

Let us now check that $v_{\eps,R}$ is a continuous solution to~\eqref{NEWktr} in the whole set $\overline{B_R\setminus K_\eps}$. Since $J_\eps\in L^2(\R^N)$ and $v_{\eps,R}\in L^2(B_R\setminus K_\eps)$, it follows from the equation~\eqref{NEWktr} satisfied by $v_{\eps,R}$ that $N_\eps(\cdot,v_{\eps,R}(\cdot))\in L^\infty(B_R\setminus K_\eps)$ where $N_\eps(x,s):=\mathcal{J}_\eps(x)\hspace{0.1em}s-g_\eps(s)$. By \eqref{Cdn:assurant:continuite}, the map $N_\eps(x,\cdot)$ is bijective and thus $v_{\eps,R}\in L^\infty(B_R\setminus K_\eps)$. Using now Lemma \ref{LEMMA:INT} and \eqref{Cdn:assurant:continuite} we may further infer that $v_{\eps,R}$ is continuous in $\overline{B_R\setminus K_\eps}$.

To complete the proof it remains to show that $0<v_{\eps,R}<1$. Let us first prove that $v_{\eps,R}<1$. Suppose, by contradiction, that $\|v_{\eps,R}\|_\infty\geq1$. Then, by continuity of $v_{\eps,R}$, there must be a point $\bar{x}\in\overline{B_R\setminus K_\eps}$ at which $v_{\eps,R}$ attains its maximum, i.e. $v_{\eps,R}(\bar{x})=\|v_{\eps,R}\|_\infty$. Using now the equation satisfied by $v_{\eps,R}$, we have
$$ 0\geq\int_{B_R\setminus K_\eps}J_\eps(\bar{x}-y)(v_{\eps,R}(y)-v_{\eps,R}(\bar{x}))\mathrm{d}y=c_\eps(\bar{x})\hspace{0.05em}v_{\eps,R}(\bar{x})-g_\eps(v_{\eps,R}(\bar{x}))\geq0. $$
Thus, since $\mathrm{supp}(J_\eps)=B_{\eps/2}$, we have $v_{\eps,R}(y)=v_{\eps,R}(\bar{x})$ for any $y\in B_{\eps/2}(\bar{x})\cap\overline{B_R\setminus K_\eps}$. Note that $B_{\eps/2}(\bar{x})\cap\overline{B_R\setminus K_\eps}$ is nonempty whence we may iterate this reasoning over again and obtain that $v_{\eps,R}\equiv v_{\eps,R}(\bar{x})\geq1$. Now choose $x_0\in \O_\eps$ such that $c_\eps(x_0)>0$. Then, evaluating  \eqref{NEWktr} at $x_0$, one obtains
$$ 0=\int_{B_R\setminus K_\eps}J_\eps(x_0-y)(v_{\eps,R}(y)-v_{\eps,R}(x_0))\mathrm{d}y=c_\eps(x_0)\hspace{0.05em}v_{\eps,R}(x_0)-g_\eps(v_{\eps,R}(x_0))\ge c_\eps(x_0)>0, $$
which is a contradiction.

Therefore $v_{\eps,R}<1$. Since, by construction, we have that $v_{\eps,R}\geq0$, it remains to check that $v_{\eps,R}$ cannot cancel. Assume, by contradiction, that this is the case, namely that there exists a point $x_0\in\overline{B_R\setminus K_\eps}$ such that $v_{\eps,R}(x_0)=0$. Then, by~\eqref{NEWktr}, we have that
$$ \int_{B_R\setminus K_\eps}J_\eps(x_0-y)(v_{\eps,R}(y)-v_{\eps,R}(x_0))\mathrm{d}y=0, $$
and, as above, this implies that $v_{\eps,R}\equiv0$. However, since $v_{\eps,R}\in\mathbb{B}_{\delta_0}(w_0)$ and $\delta_0<|B_{R_{0}}|^{1/2}$, we have $\delta_0\geq\|v_{\eps,R}-w_0\|_{L^2(B_R\setminus K_\eps)}=\|w_0\|_{L^2(B_R\setminus K_\eps)}=|B_{R_{0}}|^{1/2}>\delta_0$, which is a contradiction. The proof of Proposition~\ref{nmeClm} is thereby complete.
\end{proof}
From now on (and until the end of Section~\ref{section-supersol}), $\eps$ will be fixed and taken so small  that  $0<\eps<\bar \eps$, where $\bar \eps$ is as defined in Proposition \ref{nmeClm}.

\subsection{An extension procedure}
Let us now complete the proof of Lemma \ref{bch-lem-supersol}.
We will modify the function $v_{\eps,R}$ constructed above in order to get a continuous super-solution to \eqref{bch-eq-f} satisfying \eqref{bch-eq-f-bc}.
Let us briefly explain our strategy.
Since, by construction, $v_{\eps,R}$ satisfies \eqref{NEWktr}, the function $u_{\eps,R}=1-v_{\eps,R}$ verifies \eqref{KTR1} and,
as already noted above, extending the function $u_{\eps,R}$  by $1$ outside $B_{R}$, we obtain a (discontinuous) super-solution to ~\eqref{bch-eq-f} that satisfies \eqref{bch-eq-f-bc}.
The aim of this section is to find the right extension of $u_{\eps,R}$ that provides the desired super-solution.

To do so, we first introduce some useful notations. Given $R>0$ and $x\in\R^N$, we let $\mathcal{P}_{R}(x)$ be the projection of $x$ to the ball $\overline{B_R}$, that is
$$\mathcal{P}_{R}(x)\in \overline{B_R}\ \text{ and }\ |x-\mathcal{P}_{R}(x)|=\mathrm{dist}(x,B_R)=\min_{y\in\overline{B_R}}|x-y|.$$
For $\sigma>0$, we let $\overline{u}_{\eps,\sigma} \in C(\overline{\R^N\setminus K_\eps})$ be the following  function
\begin{equation}\label{bch-def-supersol}
\overline{u}_{\eps,\sigma}(x):=\min\big\{u_{\eps,R}(\p{R}(x))+\sigma^{-1}\,|x-\mathcal{P}_{R}(x)|,1\big\}.
\end{equation}
We shall see that, for well-chosen $\sigma$, the function $\overline{u}_{\eps,\sigma}$ will satisfy
\begin{equation}\label{bch-eps-supersol}
L_{\eps}\overline{u}_{\eps,\sigma}(x) + \tilde f_\eps( \overline{u}_{\eps,\sigma}(x))\le 0 \quad \text{for all }x \in \R^N\setminus K_\eps,
\end{equation}
where $L_\eps$ is the nonlocal operator given by
\begin{align}
L_{\eps}w(x):=\int_{\R^N\setminus K_\eps}J_\eps(x-y)(w(y)-w(x))\mathrm{d}y. \label{Loperateur:Leps}
\end{align}
Namely, we claim
\begin{claim}\label{BC:claim:fin}
There exists $\sigma_\eps>0$ such that $\overline{u}_{\eps,\sigma}$ satisfies \eqref{bch-eps-supersol} for all $0<\sigma<\sigma_\eps$.
\end{claim}
Observe that by proving Claim~\ref{BC:claim:fin}, we end the proof of Lemma \ref{bch-lem-supersol}. Indeed, by construction, we have $f\leq\widetilde{f}$ so that $\overline{u}_{\eps,\sigma}$ trivially satisfies \eqref{bch-eq-f}. As for condition \eqref{bch-eq-f-bc} it is also satisfied (by construction of $\overline{u}_{\eps,\sigma}$) provided that $R$ is taken sufficiently large.
\begin{proof}
Define $\a_R:=\R^N\setminus \overline{B_R}$. As in the previous section, we set
$$\j_\eps(x)=\int_{\R^N\setminus K_\eps}J(x-y)\hspace{0.05em} \mathrm{d}y\ \text{ and }\ c_\eps(x)=\int_{\R^N\setminus B_R} J_\eps(x-y)\hspace{0.05em} \mathrm{d}y.$$
Then, in view of \eqref{bch-def-supersol}, we have
\begin{align}\label{bch-eq-master}
L_{\eps}\overline{u}_{\eps,\sigma}(x) + \tilde f_\eps( \overline{u}_{\eps,\sigma}(x)) &\le  \int_{B_R\setminus K_\eps}J_\eps(x-y)(u_{\eps,R}(y)-\overline{u}_{\eps,\sigma}(x))\mathrm{d}y \nonumber \\
&\qquad+ c_\eps(x)(1-\overline{u}_{\eps,\sigma}(x)) + \tilde f(\overline{u}_{\eps,\sigma}(x)).
\end{align}
Since $\overline{u}_{\eps,\sigma}(x)=u_{\eps,R} (x)$ for all $x\in \overline{B_R\setminus K_\eps}$, using \eqref{KTR1} we easily get that
\begin{equation}\label{bch-eq-ss1}
L_{\eps}\overline{u}_{\eps,\sigma}(x) + \tilde f_\eps( \overline{u}_{\eps,\sigma}(x))\le 0 \quad \text{for }x\in \overline{B_R\setminus K_\eps}.
\end{equation}
To complete the proof, it remains to show that $\overline{u}_{\eps,\sigma}$ satisfies \eqref{bch-eps-supersol}  in the set $\a_R$. We shall consider two sub-domains, $\Pi^+$ and $\Pi^-$, defined as follows
$$\baa{l}
\Pi^-:=\a_R\cap\big\{\overline{u}_{\eps,\sigma}<1\big\},\vspace{3pt}\\
\Pi^+:=\a_R\cap\big\{\overline{u}_{\eps,\sigma}=1\big\}.\eaa$$
Note that since $\overline{u}_{\eps,\sigma}(x)=1$ for all $x \in \Pi^+$, it follows directly from \eqref{bch-eq-master} that
\begin{equation}\label{bch-eq-ss2}
L_{\eps}\overline{u}_{\eps,\sigma}(x) + \tilde f_\eps( \overline{u}_{\eps,\sigma}(x))=  \int_{\R^N\setminus K_\eps}J_\eps(x-y)(u_{\eps,R}(y) -1)\mathrm{d}y\le 0 \quad\text{for any } x\in \Pi^+.
\end{equation}
Thus, to conclude the proof we need only to check that \eqref{bch-eq-ss2} still holds in $\Pi^-$.
To this end, for any $x\in \Pi^-$ and any $s\in[0,1]$, we set
\begin{align}
g_R(x,s):=\j_\eps(\p{R}(x))\hspace{0.05em}s-\tilde f_\eps(s). \label{BC:def:de:gR}
\end{align}
Now, since $0<\eps<\eps_0$, it follows from Proposition \ref{LE:CONTINUE} that there exists a $\gamma>0$ such that
\begin{equation}\label{eq:gamma}
\inf_{z\in\R^N\setminus K_\eps}\min_{s\in [0,1]}\partial_s g_R(z,s)>\gamma.
\end{equation}
Next, since $J\in W^{1,1}(\R^N)$ (by~\eqref{C01}) we may set
\be\label{defdelta0}
\sigma_\eps:=\eps\hspace{0.1em}\gamma\times\left(\int_{\R^N}|\nabla J(z)|\hspace{0.05em}\mathrm{d}z\right)^{-1}>0.
\ee
Let us also set
$$s(x):=u_{\eps,R}(\p{R}(x))\ \hbox{ and }\ \tau(x):= \mathrm{dist}(x,B_R)=|x-\p{R}(x)|>0.$$
Then, $\overline{u}_{\eps,\sigma}$ rewrites $\overline{u}_{\eps,\sigma}(x)= s(x)+\sigma^{-1}\tau(x)$ and
\begin{equation}\label{w}
0<s(x)+\sigma^{-1}\tau(x)<1 \quad\text{for any }x\in\Pi^-.
\end{equation}
On the other hand, in view of \eqref{bch-def-supersol} and by definition of $L_{\eps}$, we can rewrite $L_{\eps}\overline{u}_{\eps,\sigma}(x)$ as
\begin{align*}
L_{\eps}\overline{u}_{\eps,\sigma}(x)&= L_{\eps}\overline{u}_{\eps,\sigma}(\p{R}(x)) + \int_{\R^N\setminus K_\eps}[J_\eps(x-y)-J_\eps(\p{R}(x)-y)](\overline{u}_{\eps,\sigma}(y)-\overline{u}_{\eps,\sigma}(x))\mathrm{d}y \\
&\qquad- \frac{\tau(x)}{\sigma}\int_{\R^N\setminus K_\eps}J_\eps(\p{R}(x)-y)\mathrm{d}y.
\end{align*}
Since  $\p{R}(x)\in\overline{B_R\setminus K_\eps}$, since $J\in W^{1,1}(\R^N)$ and since $J\ge 0$ a.e. in $\R^N$, by \eqref{bch-eq-ss1} we obtain
\begin{align*}
L_\eps \overline{u}_{\eps,\sigma}(x)&\le -\frac{\tau(x)}{\sigma}\j_{\eps}(\p{R}(x))-\tilde f_\eps(s(x))+ \int_{\R^N\setminus K_\eps}\left|J_\eps(x-y)-J_\eps(\p{R}(x)-y)\right|\mathrm{d}y.\\
&\le -\frac{\tau(x)}{\sigma}\j_{\eps}(\p{R}(x))-\tilde f_\eps(s(x))+ \int_{\R^N}\left|J_\eps(x-y)-J_\eps(\p{R}(x)-y)\right|\mathrm{d}y.\\
&\le -\frac{\tau(x)}{\sigma}\j_{\eps}(\p{R}(x))-\tilde f_\eps(s(x))+ \frac{\tau(x)}{\eps}\int_{\R^N}|\nabla J(z)|\hspace{0.05em}\mathrm{d}z.
\end{align*}
Therefore, we get
\begin{align*}
L_\eps \overline{u}_{\eps,\sigma}(x)+\tilde f_{\eps}(s(x)+\sigma^{-1}\tau(x))&\le \left(\tilde f_\eps(s(x)+\sigma^{-1}\tau(x)) -\tilde f_\eps(s(x))\right) \\
&\quad -\frac{\tau(x)}{\sigma}\j_{\eps}(\p{R}(x))+\frac{\tau(x)}{\eps}\int_{\R^N}|\nabla J(z)|\hspace{0.05em}\mathrm{d}z.
\end{align*}
By adding and subtracting $s(x)\j_{\eps}(\p{R}(x))$ on the right hand side of the above inequality and recalling \eqref{BC:def:de:gR}, we obtain
$$
L_\eps \overline{u}_{\eps,\sigma}(x)+\tilde f_{\eps}(\overline{u}_{\eps,\sigma}(x))\le \left(g_R\left(x,s(x)\right)-g_R\left(x,s(x)+\sigma^{-1}\tau(x)\right)\right)+ \gamma\hspace{0.1em}\sigma_\eps^{-1}\tau(x).
$$
where we have used~\eqref{defdelta0}. By~\eqref{eq:gamma}, \eqref{w} and the mean value theorem, we deduce that there exists some
$$ \xi\in\left[s(x),s(x)+\sigma^{-1}\tau(x)\right]\subset[0,1], $$
such that
$$ g_R\left(x,s(x)\right)-g_R\left(x,s(x)+\sigma^{-1}\tau(x)\right)=-\partial_s g_R(x,\xi)\hspace{0.1em}\sigma^{-1}\tau(x)\leq-\gamma\hspace{0.1em}\sigma^{-1}\tau(x). $$
Therefore, for every $0<\sigma<\sigma_\eps$, we obtain that
\begin{align*}
L_\eps \overline{u}_{\eps,\sigma}(x)+\tilde f_{\eps}(\overline{u}_{\eps,\sigma}(x))&\le \gamma\hspace{0.05em}\tau(x)\left(\frac{1}{\sigma_\eps}-\frac{1}{\sigma}\right)<0 \quad\text{for any }x\in\Pi^-.
\end{align*}
The proof of Claim~\ref{BC:claim:fin} is thereby complete.
\end{proof}
\begin{rem}\label{RK:geo3}
An analogue version of Lemma \ref{bch-lem-supersol} holds when $J_\eps(x-y)$ is replaced by $\widetilde{J}_\eps(d_g(x,y))$ where $\widetilde{J}_\eps$ is a locally integrable function such that $\widetilde{J}_\eps(|z|)=J_\eps(z)$ and $d_{\mathrm{g}}(x,y)$ is the geodesic distance on $\overline{\R^N\setminus K_\eps}$.
Indeed, the only places where the structure of the radial kernel $J_\eps$ came into place is when we used the Poincar\'e-type inequality \cite[Theorem 1.1]{Ponce2004} in Proposition \ref{Poincaree}, when we asserted that the solutions to \eqref{NEWktr} satisfying $\max_{[0,1]}f_\eps'<\inf_{B_R\setminus K_\eps}\mathcal{J}_\eps$ are continuous and when we made our extension procedure. But the Poincar\'e inequality was only needed in the ball $B_{R_0}$ and, by convexity, it trivially holds that $J_\eps(x-y)=\widetilde{J}_\eps(d_{\mathrm{g}}(x,y))$ for any $(x,y)\in B_{R_0}\times B_{R_0}$. Similarly, the extension procedure required only to evaluate the new function on the annulus $B_{R+\sigma}\setminus B_R$ but, since $R-R_1>0$ is large and $\eps$ is small, it still holds that $J_\eps(x-y)=\widetilde{J}_\eps(d_{\mathrm{g}}(x,y))$ for any $x \in B_{R+\sigma}\setminus B_R $ and any $y\in \R^N\setminus K_\eps$. Moreover, as already noted in Remark \ref{rem-geod}, condition \eqref{LE:geo:cont} still implies the continuity of solutions to the corresponding auxiliary problem:
$$\int_{B_R\setminus K_\eps}\widetilde{J}_\eps(d_{\mathrm{g}}(x,y))(v_{\eps,R}(y)-v_{\eps,R}(x))\mathrm{d}y -\widetilde{c}_{\eps}(x)\hspace{0.05em}v_{\eps,R}+g_\eps(v_{\eps,R}(x))=0 \quad \text{for }x\in \overline{B_R\setminus K_\eps},$$
where, by analogy, we have set
$$\widetilde{c}_{\eps}(x):=\int_{\R^N\setminus B_R} \widetilde{J}_\eps(d_{\mathrm{g}}(x,y))\hspace{0.05em}\mathrm{d}y.$$
In fact, the only place where some care should be taken is when justifying that if
\begin{align}
\int_{B_R\setminus K_\eps}\widetilde{J}_\eps(d_{\mathrm{g}}(\bar{x},y))(v_{\eps,R}(y)-v_{\eps,R}(\bar{x}))\mathrm{d}y=0, \label{geo:iter}
\end{align}
where $\bar{x}\in\overline{B_R\setminus K_\eps}$ is a point at which $v_{\eps,R}$ reaches an extremum, then it holds that $v_{\eps,R}(y)\equiv v_{\eps,R}(\bar{x})$ for any $y\in\overline{B_R\setminus K_\eps}$ (which is needed to establish the analogue of Proposition~\ref{nmeClm}). But, fortunately, the geometry of $K_\eps$ is simple enough to ensure that this is still the case. Indeed, \eqref{geo:iter} implies that $v_{\eps,R}(y)\equiv v_{\eps,R}(\bar{x})$ for any $y\in\Pi_1(\bar{x}):=\{z\in\overline{B_R\setminus K_\eps}; d_\mathrm{g}(\bar{x},z)<\eps/2\}$. By iteration, one finds that $v_{\eps,R}(y)\equiv v_{\eps,R}(\bar{x})$ for any $y\in\Pi_j(\bar{x})$ and any $j\geq1$, where $\Pi_j(\bar{x})$ is given by
$$ \Pi_{j+1}(\bar{x}):=\bigcup_{y\in\Pi_j(\bar{x})}\left\{z\in\overline{B_R\setminus K_\eps};~d_{\mathrm{g}}(y,z)< \eps/2\right\}, \hbox{ for any }j\geq1. $$
Then, one can show that, for some $j_0\geq1$ (independent of $\bar{x}$), it holds that $B_{\eps/4}(\bar{x})\cap\overline{B_R\setminus K_\eps}\subset\Pi_{j_0}(\bar{x})$. Whence, iterating the same reasoning over again, one gets that $v_{\eps,R}(y)\equiv v_{\eps,R}(\bar{x})$ for any $y\in B_{k\eps/4}(\bar{x})\cap\overline{B_R\setminus K_\eps}$ and any $k\in\N$; which then gives the desired result.
\end{rem}
\section{Construction of continuous global solutions}\label{section-proof}

In this final section we construct a positive nonconstant solution to \eqref{eq-eps}. Our goal will be to find an ordered pair of global continuous sub- and super-solution. That is, given $0<\eps<\eps^*$ (where $\eps^*$ has the same meaning as in Lemma~\ref{bch-lem-supersol}), we aim to construct two functions, $\underline{u}_\eps$ and $\overline{u}_\eps$, such that
$$\left\{
\begin{array}{rl}
L_\eps \overline{u}_\eps+f_\eps(\overline{u}_\eps)\leq0 & \text{in }\R^N\setminus K_\eps, \vspace{3pt}\\
L_\eps\underline{u}_\eps+f_\eps(\underline{u}_\eps)\geq0 & \text{in }\R^N\setminus K_\eps, \vspace{3pt}\\
0\leq \underline{u}_\eps\le \overline{u}_\eps\leq 1 & \text{in }\R^N\setminus K_\eps,
\end{array}
\right.$$
(where $L_\eps$ is as in \eqref{Loperateur:Leps}) and which further satisfy
\begin{align}
\lim_{x_1\to+\infty}\,\underline{u}_\eps(x)=1 \quad \text{and} \quad \lim_{|x|\to+\infty}\,\overline{u}_\eps(x)=1. \label{limitesSS}
\end{align}
Here, $x_1=x\cdot e_1$ where $e_1:=(1,0,\cdots,0)\in\mathbb{S}^{N-1}$. Then, by Lemmata \ref{Lem-iter} and \ref{LEMMA:INT} we automatically obtain the existence of a continuous solution $u_\eps$ to
\begin{align}
L_\eps u_\eps+f_\eps(u_\eps)=0 \quad\hbox{in }\R^N\setminus K_\eps, \label{BC:EQUATION}
\end{align}
satisfying $0\leq \underline{u}_\eps\leq u_\eps\leq \overline{u}_\eps\leq 1$. This, together with~\eqref{limitesSS}, yields a continuous solution to \eqref{BC:EQUATION} satisfying $0<u_\eps<1$ and $u_\eps(x)\to1$ as $x_1\to\infty$. In particular, we have $\sup_{x\in\R^N\setminus K_\eps}\,u_\eps(x)=1$. Since~\eqref{C01},~\eqref{C02} are satisfied, $u_\eps$ is continuous, $J_\eps$ is compactly supported and $J_\eps\in L^2(\R^N)$ (by \eqref{J-assum}), we may apply Lemma \ref{LEMMA:Behave} and we obtain that $\lim_{|x|\to+\infty}\,u_\eps(x)=1$, which proves that $u_\eps$ satisfies the requirements of Theorem~\ref{TH:eps} and thus Theorem \ref{TH:COUNTEREXAMPLE} is proved.

Therefore, to complete the proof of Theorem~\ref{TH:eps}, we need only to prove the following lemma.

\begin{lemma}\label{BC:soussupersol}
Let $(J,f)$ be a pair satisfying \eqref{C02} and \eqref{J-assum}. Let $(K_\eps)_{0<\eps<1}$ be the family of obstacles associated to the pair $(J,f)$ (as defined in Section~\ref{section-data}). Let $(J_\eps,f_\eps)$ be as in \eqref{Donnees:fJ} and let $\eps^*>0$ be as in Lemma~\ref{bch-lem-supersol}. Then, there exists $r_0>0$ such that, for all $0<\eps<\eps^*$, there is
\begin{enumerate}
\item[(i)] a continuous global sub-solution $\underline{u}_\eps$ to~\eqref{BC:EQUATION}  satisfying $\underline{u}_\eps\equiv0$ in $\{x_1\leq r_0\}$ and $\underline{u}_\eps(x)\to 1$ as $x_1\to\infty$,
\item[(ii)] a continuous global nonconstant super-solution $\overline{u}_\eps$ to~\eqref{BC:EQUATION}  satisfying $\overline{u}_\eps\equiv 1$ in $\R^N\setminus B_{r_0}$ and $0<\overline{u}_\eps\leq 1$.
\end{enumerate}
In particular, $0\leq\underline{u}_\eps< \overline{u}_\eps\leq1$.
\end{lemma}

\begin{proof}
By Lemma~\ref{bch-lem-supersol}, we know that there exists some $R^*>0$ and some $0<\eps^*<1$ such that, for all $0<\eps<\eps^*$, there is a nonconstant super-solution $\overline{u}_\eps\in C(\overline{\R^N\setminus K_\eps})$ to \eqref{BC:EQUATION} that satisfies $\overline{u}_\eps\equiv 1$ in $\R^N\setminus B_{R^*}$. So, we are left to prove that there exists a sub-solution $\underline{u}_\eps$ to \eqref{BC:EQUATION} satisfying  (i) and such that $\underline{u}_\eps\le \overline{u}_\eps$.

To do so, let us extend $f$ outside $[0,1]$ by $f'(0)s$ when $s\ge 0$ and $f'(1)(s-1)$ for $s\ge 1$. For simplicity, we still denote by $f$ this extension. Now, we take $\delta\in(0,1)$ and we let $f_\delta$ be a $C^1$ function defined in $\R$ such that
$$\left\{\begin{array}{l}
f_\delta\leq f\text{ in }\R,\text{ and }f_\delta(s)=f(s)\text{ for }s\ge \theta, \vspace{3pt}\\
f_\delta\text{ has only one zero, }\theta_\delta=\theta,\text{ in }(-\delta,1), \vspace{3pt}\\
f_\delta\left(-\delta\right)=0,~~f_\delta\left(1\right)=0, \vspace{3pt}\\
f_\delta'(s)<1\text{  for any }s\in[-\delta,1]\text{ and }f_\delta'\left(-\delta\right),\,f_\delta'\left(1\right)<0, \vspace{3pt}\\
\ds\int_{-\delta}^{1}f_\delta(s)\hspace{0.05em}\mathrm{d}s>0.
\end{array}\right.$$
Since $f\in C^1(\R)$ satisfies \eqref{C02} such a function $f_\delta\in C^1\big(\R\big)$ always exists provided that $\delta$ is taken sufficiently small, say if $0<\delta<\delta_1$ for some small $\delta_1>0$.

Let $f_{\eps,\delta}(s):=\eps^2f_\delta(s)$ and let $L_{\R^N}$ be the operator given by
\begin{align}
L_{\R^N}u(x):=\int_{\R^N}J_\eps(x-y)(u(y)-u(x))\mathrm{d}y. \label{LRN}
\end{align}
Since $J_\eps$ is radially symmetric (because $J$ is), using the results obtained in \cite{Bates,Chen,Coville,Yagisita}, we know that, for any $0<\eps<1$, there exists an increasing function $\phi_{\eps,\delta}\in C^1(\R)$ and a number $c_{\eps,\delta}>0$ such that the function $\varphi_{\eps,\delta}(x):=\phi_{\eps,\delta}(x\cdot e_1) $ satisfies
\begin{align}
\left\{
\begin{array}{rl}
L_{\R^N}\varphi_{\eps,\delta}(x)+f_{\eps,\delta}(\varphi_{ \eps,\delta}(x))=c_{\eps,\delta}\phi_{\eps,\delta}'(x_1)\ge 0&\text{for all }x\in \R^N, \vspace{3pt}\\
\varphi_{\eps,\delta}(-\infty)=-\delta,\ \ \varphi_{\eps,\delta}(\infty)=1~\text{ and }~\varphi_{\eps,\delta}=0& \text{in }H_{e_1},
\end{array}
\right. \label{AUXILIARY2}
\end{align}
where $H_{e_1}$ is the hyperplane $H_{e_1}:=\{x_1=0\}$.
Now, for any $r_0>0$, we let $\varphi_{\eps,\delta,r_0}$ be the function defined by
$$ \varphi_{\eps,\delta,r_0}(x):=\varphi_{\eps,\delta}(x-r_0). $$
By construction, for every $r_0>0$, we have
\begin{align}
L_{\R^N}\varphi_{\eps,\delta,r_0}+f_\eps(\varphi_{\eps,\delta,r_0})\geq L_{\R^N}\varphi_{\eps,\delta,r_0}+f_{\eps,\delta}(\varphi_{\eps,\delta,r_0})\geq0 \quad{\mbox{in }}\R^N, \label{lrnphi}
\end{align}
Now, we set
$$ \underline{u}_\eps(x):=\max\big\{0,\varphi_{\eps,\delta,r_0}(x)\big\}\quad\text{and}\quad H_*:=\big\{x\in\R^N\,;\ x_1\geq r_0\big\}. $$
Note that, for all $0<\eps<\eps^*$, it holds that $K_\eps\subset\R^N\setminus H_*$ provided that $r_0$ is chosen sufficiently large. Let us now prove that, for $r_0$ large enough, $\underline{u}_\eps$ is a sub-solution to \eqref{BC:EQUATION}.

First, if $x\in \R^N\setminus(K\cup H_*)$, then $\underline{u}_{\eps}(x)=0$ and
\begin{align}
L_\eps\underline{u}_\eps(x)+f_\eps(\underline{u}_\eps(x))=\int_{\R^N\setminus K_\eps}J_\eps(x-y)\underline{u}_\eps(y)\hspace{0.05em}\mathrm{d}y\geq 0. \label{lrnphi2}
\end{align}
Next, if $x\in H_*$, then, since $J_\eps$ is compactly supported, we have
$$ \bigcup_{x\in H_*}\big( x+\mathrm{supp}(J_\eps) \big) \subset \R^N\setminus K_\eps, $$
provided that $r_0$ is chosen sufficiently large. From this and~\eqref{lrnphi}, we deduce that
$$\baa{rcl}
L_\eps\underline{u}_\eps(x)+f_\eps(\underline{u}_\eps(x)) & = & \ds\int_{\R^N\setminus K_\eps}J_\eps(x-y)(\underline{u}_\eps(y)-\varphi_{\eps,\delta,r_0}(x))\mathrm{d}y+f_\eps(\varphi_{\eps,\delta,r_0}(x))\vspace{3pt}\\
& \geq & \ds\int_{\R^N\setminus K_\eps}J_\eps(x-y)(\varphi_{\eps,\delta,r_0}(y)-\varphi_{\eps,\delta,r_0}(x))\mathrm{d}y +f_\eps(\varphi_{\eps,\delta,r_0}(x))\vspace{3pt}\\
& = & L_{\R^N}\varphi_{\eps,\delta,r_0}(x)+f_\eps(\varphi_{\eps,\delta,r_0}(x))\geq0.\eaa$$
Together with~\eqref{lrnphi2}, we obtain that $\underline{u}_\eps$ is a global sub-solution to~\eqref{BC:EQUATION} which, by~\eqref{AUXILIARY2}, satisfies $\underline{u}_\eps(x)\to1$ as $x_1\to\infty$ and $\underline{u}_\eps(x)=0$ if $x_1\leq r_0$. By increasing $r_0$ to $R^*$ (if necessary) we then achieve $\underline{u}_\eps < \overline{u}_\eps$ when $0<\eps<\eps^*$. The proof of Lemma~\ref{BC:soussupersol} is thereby complete.
\end{proof}
\begin{rem}
Observe that, on account of Remarks \ref{rem-geod}, \ref{RK:geo2} and \ref{RK:geo3}, the \emph{same} proof as above yields an analogous result with $L_{\mathrm{g}}$ in place of $L$. To see this, it suffices to notice that our arguments are essentially focused on what is happening \emph{far away} from $K$ and, since the kernel we consider is compactly supported, the operator $L_{\mathrm{g}}$ will then coincide with $L$ (possibly up to take $R$ sufficiently large). In like manner, as already mentioned in Remark \ref{rem-geod}, the fact that ``$\sup_{\R^N\setminus K_\eps}\,u=1$" implies that ``$\lim_{|x|\to\infty}\,u(x)=1$" still holds with $L_{\mathrm{g}}$ in place of $L$ since, here as well, the proof relies only on estimates of the behaviour of $u$ far away from $K_\eps$.
\end{rem}

\section*{Appendix}

In this appendix, we prove Lemma~\ref{Lem-iter}. Our strategy closely follows \cite{BCHV,CovDup07} and relies on the well-known monotone iterative method. Before doing so, we first state a preliminary lemma.
\begin{lemma}\label{LE:comparaison}
Let $K\subset\R^N$ be a compact set and assume that $J$ satisfies \eqref{C01}. Let $k>0$ and let $w\in C(\R^N\setminus K)$ be such that
\begin{align}
Lw-kw\geq0\quad\text{in }\R^N\setminus K, \label{Lemme:prem}
\end{align}
and that
\begin{align}
\limsup_{|x|\to\infty}\,w(x)\leq0.  \label{Lemme:deux}
\end{align}
Then,
$$ w\leq 0\quad\text{in }\R^N\setminus K. $$
\end{lemma}
\begin{proof}
Suppose, by contradiction, that $\sup_{\R^N\setminus K}w>0$. Then, by assumption \eqref{Lemme:deux}, there exists a number $r>0$ with $K\subset B_r$ and a sequence $(x_j)_{j\geq0}\subset B_r\setminus K$ such that
\begin{align}
\lim_{j\geq0}w(x_j)=\sup_{B_r\setminus K}w=\sup_{\R^N\setminus K} w>0. \label{Lem:Ap:limittt}
\end{align}
Since $(x_j)_{j\geq0}$ is bounded, up to extraction of a subsequence, there exists a point $\bar{x}\in\overline{B_r\setminus K}$ such that $x_j\to\bar{x}$ as $j\to\infty$. Moreover, since $w$ is continuous and \eqref{Lemme:prem} is satisfied everywhere in $\R^N\setminus K$, it makes sense to evaluate \eqref{Lemme:prem} at $x_j$ for any $j\geq0$. That is, we have
$$ \int_{\R^N\setminus K}J(x_j-y)(w(y)-w(x_j))\mathrm{d}y\geq k\hspace{0.05em}w(x_j)\quad\text{for any }j\geq0. $$
But, since $k>0$, using \eqref{Lem:Ap:limittt} and the dominated convergence theorem, we obtain
\begin{align*}
0&\geq \int_{\R^N\setminus K}J(\bar{x}-y)\left(w(y)-\sup_{\R^N\setminus K}w\right)\mathrm{d}y\geq k\sup_{\R^N\setminus K} w>0,
\end{align*}
which is a contradiction. The proof is thereby complete.
\end{proof}
We are now in position to prove Lemma~\ref{Lem-iter}.
\begin{proof}[Proof of Lemma~\ref{Lem-iter}]
Let us first observe that, from the assumptions made on $J$, the operator $L$ is linear and continuous on $(C_0(\R^N\setminus K),\left\|\cdot\right\|_\infty)$, where
$$ C_0(\R^N\setminus K):=\left\{w\in C(\R^N\setminus K); \lim_{|x|\to\infty}w(x)=0\right\}. $$
Indeed, this is because, given any $w\in C_0(\R^N\setminus K)$, we have
$$ Lw(x)=\int_{\R^N}J(y)\left(\mathds{1}_{x-\R^N\setminus K}(x)\hspace{0.05em}w(x-y)\right)\mathrm{d}y-\mathcal{J}(x)w(x), $$
where $\mathcal{J}$ is as in \eqref{defJ(x)}, and, by the dominated convergence theorem, we have that $Lw(x)\to0$ as $|x|\to\infty$. The continuity of $Lw$ is a mere consequence of the continuity of translations in $L^1(\R^N)$ and of the continuity of $w$, as is easily seen from the (trivial) inequality
\begin{align}
|Lw(x_1)-Lw(x_2)|\leq 2\|w\|_\infty\int_{\R^N}|J(y+x_1-x_2)-J(y)|\mathrm{d}y+|w(x_1)-w(x_2)|, \label{BC:continu:Lw}
\end{align}
which holds for any $x_1,x_2\in\R^N\setminus K$. So that $L$ indeed maps $C_0(\R^N\setminus K)$ into itself. Moreover, the continuity of the operator $L$ follows from the fact that
$$ \|Lw\|_\infty\leq 2\|w\|_\infty \quad\text{for any }w\in C_0(\R^N\setminus K). $$

Next, we let $k>0$ be a number large enough so that the map $s\mapsto -ks-f(s)$ is decreasing in $[0,1]$ and that $k\in\rho(L)$ where $\rho(L)$ denotes the resolvent of the operator $L$.

Let $\underline{u}$ and $\overline{u}$ be continuous global sub- and super-solutions to
\begin{align}
Lu+f(u)=0\quad\text{in }\R^N\setminus K, \label{LEQUATION:pb:iter:mon}
\end{align}
satisfying \eqref{SousSur:iter1} and \eqref{SousSur:iter2}.

We will construct a solution $u$ to \eqref{LEQUATION:pb:iter:mon} satisfying $\underline{u}\leq u\leq\overline{u}$ using a monotone iterative scheme. That is, we will construct $u$ as the limit of an appropriate sequence of functions. The main tool behind our construction is the comparison principle Lemma~\ref{LE:comparaison}. To this end, we have to make sure that the sequence we construct has the right asymptotic behavior as $|x|\to\infty$ (as required by Lemma~\ref{LE:comparaison}). With this aim in mind, we first construct an appropriate sequence of auxiliary functions. Namely, we define $v_0\equiv 0$ and, for $j\geq0$, we let
\begin{align}
L v_{j+1}(x)-k\hspace{0.05em}v_{j+1}(x)=-k\hspace{0.05em}v_{j}(x)-f(\overline{u}(x)+v_j(x))-L\overline{u}(x)\quad\text{for }x\in\R^N\setminus K. \label{suite:auxiliaire:schema:iter}
\end{align}
Let us check that the $v_j$'s are well-defined elements of $C_0(\R^N\setminus K)$. Since
$k\in \rho(L)$ and $0\equiv v_0\in C_0(\R^N\setminus K)$, $v_1$ is a well-defined element of $C_0(\R^N\setminus K)$ as soon as $$ f(\overline{u}(\cdot))+L\overline{u}(\cdot) \in C_0(\R^N\setminus K),$$
which is the case since  $f(1)=0$, $f$ is continuous, $\overline{u}(x)\to1$ as $|x|\to\infty$ and $L\overline{u}\in C_0(\R^N\setminus K)$ (because $\overline{u}\in C(\R^N\setminus K)$) and
$$ L\overline{u}(x)=\int_{\R^N}J(y)\hspace{0.05em}\mathds{1}_{x-\R^N\setminus K}(y)(\overline{u}(x-y)-\overline{u}(x))\mathrm{d}y\underset{|x|\to\infty}{\longrightarrow}\int_{\R^N}J(y)(1-1)\mathrm{d}y=0.$$
Similarly, if, for some $j\geq0$, it holds that $v_j\in C_0(\R^N\setminus K)$, then, given that $k\in\rho(L)$ and that $L\overline{u}\in C_0(\R^N\setminus K)$, $v_{j+1}$ is a well-defined element of $C_0(\R^N\setminus K)$ as soon as
$$ f(\overline{u}(\cdot)+v_j(\cdot)) \in C_0(\R^N\setminus K),$$
which trivially holds  since  $f$ is continuous, $f(1)=0$ and $\overline{u}(x)\to1, v_j(x)\to 0$ as $|x|\to\infty$. Whence, by induction, we infer that the $v_j$'s are, indeed, well-defined elements of $C_0(\R^N\setminus K)$.
Let us now define a sequence $(u_j)_{j\geq0}\subset C(\R^N\setminus K)$ by setting $u_j:=\overline{u}+v_j$. Then, by construction, we have
\begin{align}
Lu_{j+1}(x)-k\hspace{0.05em}u_{j+1}(x)=-k\hspace{0.05em}u_j(x)-f(u_j(x)) \quad\text{for any }x\in\R^N\setminus K\text{ and }j\geq0, \label{BC:eq:generale:iter}
\end{align}
and the $u_j$'s satisfy the limit condition
\begin{align}
\lim_{|x|\to\infty}u_j(x)=1 \quad\text{for any }j\geq0. \label{App:comportement:asympt}
\end{align}
We will show that the desired solution to \eqref{LEQUATION:pb:iter:mon} can be obtained as the pointwise limit of $(u_j)_{j\geq0}$. Let us proceed step by step. First, when $j=0$, we have
\begin{align}
Lu_1(x)-k\hspace{0.05em}u_1(x)=-k\hspace{0.05em}u_0(x)-f(u_0(x)) \quad\text{for }x\in\R^N\setminus K. \label{iter:u1u0}
\end{align}
We claim that $\underline{u}\leq u_1\leq u_0=\overline{u}$ in $\R^N\setminus K$. Indeed, we have
\be\left\{\begin{array}{rcl}
L(u_1-u_0)(x)-k(u_1-u_0)\!\! & \!\!=\!\! &\!\! -Lu_0(x)-f(u_0(x)), \vspace{3pt}\\
L(u_{1}-\underline{u})(x)-k(u_{1}-\underline{u})\!\! &  \!\!\leq\!\! &\!\! f(\underline{u}(x))+k\hspace{0.05em}\underline{u}(x)-f(u_0(x))-k\hspace{0.05em}u_0(x).\end{array}\right.  \nonumber
\ee
Since $u_0=\overline{u}$ is a super-solution to \eqref{LEQUATION:pb:iter:mon}, $\underline{u}\leq \overline{u}$ and $s\mapsto-k\hspace{0.05em}s-f(s)$ is decreasing, we obtain that
\begin{align}
\left\{
\begin{array}{rl}
L(u_1-u_0)(x)-k(u_1-u_0)\geq 0,& \vspace{3pt} \\
L(u_{1}-\underline{u})(x)-k(u_{1}-\underline{u})\leq 0.&
\end{array}
\right.\label{sys:iter:u1u0}
\end{align}
By construction of $u_1$ (remember \eqref{SousSur:iter1} and \eqref{App:comportement:asympt}), we have
\begin{align}
\lim_{|x|\to\infty}(u_1-u_0)(x)=0~~\text{and}~~\liminf_{|x|\to\infty}\,(u_{1}-\underline{u})(x)\geq0. \label{LaLimiteU01}
\end{align}
This, together with Lemma~\ref{LE:comparaison}, then gives that $\underline{u}\leq u_1\leq u_0=\overline{u}$ in $\overline{\R^N\setminus K}$. Similarly, by \eqref{BC:eq:generale:iter}, the function $u_2\in C(\R^N\setminus K)$ solves \eqref{iter:u1u0} with $u_2$ in place of $u_1$ and $u_1$ in place of $u_0$. Thus, from \eqref{SousSur:iter1}, \eqref{App:comportement:asympt} and the monotonicity of $s\mapsto-k\hspace{0.05em}s-f(s)$, we deduce that \eqref{sys:iter:u1u0} and \eqref{LaLimiteU01} still hold with $u_2$ instead of $u_1$ and $u_1$ instead of $u_0$. We may then apply the comparison principle Lemma~\ref{LE:comparaison} and we deduce that $\underline{u}\leq u_2\leq u_1\leq u_0=\overline{u}$ in $\R^N\setminus K$. By induction, we infer that the $u_j$'s satisfy the monotonicity relation
$$ \underline{u}\leq \cdots\leq u_{j+1}\leq u_j\leq \cdots\leq u_2\leq u_1\leq u_0=\overline{u}. $$
Since $(u_j)_{j\geq0}$ is non-increasing and bounded from below by $\underline{u}$, the function
\begin{align}
u(x):=\lim_{j\to\infty}u_j(x)\in\left[\underline{u}(x),\overline{u}(x)\right], \label{definition:de:la:sol}
\end{align}
is well-defined for any $x\in\R^N\setminus K$. In particular, since $0\leq\underline{u}\leq\overline{u}\leq1$, it follows from \eqref{definition:de:la:sol} that $u\in L^\infty(\R^N\setminus K)$. It remains only to check that the function $u$ is a solution to \eqref{LEQUATION:pb:iter:mon}. For it, it suffices to let $j\to\infty$ in \eqref{BC:eq:generale:iter} (using the dominated convergence theorem), which then gives
$$ Lu(x)+f(u(x))=0 \quad\text{for any }x\in\R^N\setminus K. $$
The proof is thereby complete.
\end{proof}
\begin{rem}\label{geod-montone-iterative}
The same arguments also apply when the operator $L$ is replaced by $L_{\mathrm{g}}$ provided that $J=\widetilde{J}(\left|\cdot\right|)$ satisfies \eqref{J-assum}, since it still holds that if $w(x)\to \ell\in\R$ as $|x|\to\infty$, then $L_{\mathrm{g}}w(x)\to 0$ as $|x|\to\infty$. Moreover, the continuity of $w$ still implies the continuity of $L_{\mathrm{g}}w$ but the proof is less obvious since one can no longer rely on the continuity of translations in $L^1(\R^N)$. For the sake of completeness, we state a last lemma below which justifies why this is true.
\end{rem}

\begin{lemma}\label{Lgestcontinue}
Let $K\subset\R^N$ be a compact set and assume that $\widetilde{J}$ satisfies \eqref{J:geo} and that $\widetilde{J}$ is supported in $[0,r]$ for some $r>0$. Let $w\in C(\R^N\setminus K)$. Then, $L_{\mathrm{g}}w\in C(\R^N\setminus K)$.
\end{lemma}
\begin{proof}
Let $x_1,x_2\in\R^N\setminus K$ with $x_1$ fixed and $x_2$ arbitrarily close to $x_1$. For $w\in C(\R^N\setminus K)$, the analogue of \eqref{BC:continu:Lw} is here:
$$ |L_{\mathrm{g}}w(x_1)-L_{\mathrm{g}}w(x_2)|\leq 2\|w\|_\infty\left|\int_{\R^N}[\widetilde{J}(d_\mathrm{g}(x_1,y))-\widetilde{J}(d_{\mathrm{g}}(x_2,y))]\mathrm{d}y\right|+\|\widetilde{\mathcal{J}}\|_\infty|w(x_1)-w(x_2)|, $$
where $\widetilde{\mathcal{J}}$ is as in \eqref{defJtilde(x)}.
Since $w\in C(\R^N\setminus K)$, the delicate part is to show that the first term on the right-hand side vanishes as $x_2\to x_1$. This can be done as follows.
Let $\delta>0$ be small enough so that $x_2\in B_{\delta/2}(x_1)\subset B_{\delta}(x_1)\subset\R^N\setminus K$. Then, we may write
\begin{align*}
&\left|\int_{\R^N\setminus K}[\widetilde{J}(d_\mathrm{g}(x_1,y))-\widetilde{J}(d_{\mathrm{g}}(x_2,y))]\mathrm{d}y\right|\\
&\qquad\leq \int_{\R^N\setminus (B_\delta(x_1)\cup K)}|\widetilde{J}(d_\mathrm{g}(x_1,y))-\widetilde{J}(d_{\mathrm{g}}(x_2,y))|\mathrm{d}y+\int_{B_{\delta(x_1)}}|\widetilde{J}(d_\mathrm{g}(x_1,y))-\widetilde{J}(d_{\mathrm{g}}(x_2,y))|\mathrm{d}y \\
&\qquad=:I_1(x_1,x_2)+I_2(x_1,x_2).
\end{align*}
Since $d_{\mathrm{g}}(x_i,y)=|x_i-y|$ for any $i\in\{1,2\}$ and $y\in B_\delta(x_1)$, we have
\begin{align*}
I_2(x_1,x_2)\leq \|J(\cdot+x_1-x_2)-J\|_{L^1(\R^N)}\underset{x_2\to x_1}{\longrightarrow}0. 
\end{align*}
On the other hand, since $J$ is radially symmetric, $\mathrm{supp}(J)=B_r$ and $J\in W^{1,1}(B_r)$, by \cite[Theorems 1.1 and 2.3]{Figueir}, we have that $\widetilde{J}\in W^{1,1}((0,r),t^{N-1})$, $\widetilde{J}$ is almost everywhere equal to a continuous function, $\widetilde{J}'$ exists almost everywhere and
\begin{align}
\int_0^r|\widetilde{J}'(t)|t^{N-1}\mathrm{d}t\leq C_1\int_{B_r}|\nabla J(z)|\mathrm{d}z. \label{BC:radialSObolev}
\end{align}
Therefore, using the fact that $d_{\mathrm{g}}(x_i,y)\geq|x_i-y|\geq\delta/2$ for any $y\in\R^N\setminus(B_\delta(x_1)\cup K)$, we have
\begin{align}
I_1(x_1,x_2)&\leq\int_{\R^N\setminus (B_\delta(x_1)\cup K)}\int_{d_{\mathrm{g}}(x_2,y)}^{d_\mathrm{g}(x_1,y)}|\widetilde{J}'(t)|\mathrm{d}t\mathrm{d}y \nonumber \\
&\leq \left(\frac{2}{\delta}\right)^{N-1}\int_{\R^N\setminus (B_\delta(x_1)\cup K)}\int_{d_{\mathrm{g}}(x_2,y)}^{d_\mathrm{g}(x_1,y)}|\widetilde{J}'(t)|t^{N-1}\mathrm{d}t\mathrm{d}y. \label{LaPartieDueAI1}
\end{align}
Now, since $x_1,x_2\in B_{\delta/2}(x_1)\subset\R^N\setminus K$ and $d_{\mathrm{g}}(\cdot,\cdot)$ is a distance, we have
$$ |d_{\mathrm{g}}(x_1,y)-d_{\mathrm{g}}(x_2,y)|\leq d_{\mathrm{g}}(x_1,x_2)=|x_1-x_2|\underset{x_2\to x_1}{\longrightarrow}0 $$
Therefore, using \eqref{BC:radialSObolev}, \eqref{LaPartieDueAI1} and the dominated convergence theorem, we obtain that
$$I_1(x_1,x_2)\to0\quad\text{as }x_2\to x_1.$$
This completes the proof.
\end{proof}

\medskip

\noindent \textbf{Acknowledgement.}
This work has been carried out in the framework of Archim\`ede Labex (ANR-11-LABX-0033) and of the A*MIDEX project (ANR-11-IDEX-0001-02), funded by the ``Investissements d'Avenir" French Government program managed by the French National Research Agency (ANR). The research leading to these results has also received funding from the ANR DEFI project NONLOCAL (ANR-14-CE25-0013) and the ANR JCJC project MODEVOL (ANR-13-JS01-0009).

\end{document}